\documentclass[11pt]{article}

\usepackage{setspace}

\usepackage[margin=1in]{caption}
\usepackage{graphicx}
\usepackage{subfigure}
\usepackage{float}
\usepackage{enumitem}
\usepackage[bottom]{footmisc}
\usepackage{appendix}
\graphicspath{ {images/} }
\usepackage{amsmath,amsthm,amssymb}
\usepackage{mathrsfs}
\numberwithin{equation}{section}
\usepackage{subdepth}
\usepackage{bm}
\usepackage[margin=1in]{geometry}

\usepackage{cite}
\usepackage{hyperref}
\hypersetup{colorlinks=true,citecolor=blue, urlcolor=cyan}

	\usepackage{environ}
	\usepackage{xparse}
	\usepackage[numbers,sort]{natbib}
	\theoremstyle{definition}

    \newtheorem{remark}{Remark}[section]
	\theoremstyle{theorem}
    \newtheorem{corollary}{Corollary}[section]
	\newtheorem{theorem}{Theorem}[section]
    \newtheorem{lemma}{Lemma}[section]
    \newtheorem{proposition}{Proposition}[section]  
	\theoremstyle{remark}
	
    \usepackage{thmtools}
    \usepackage[capitalise]{cleveref}

\newcommand\numberthis{\addtocounter{equation}{1}\tag{\theequation}}
\newcommand{\R}{\ensuremath{\mathbb{R}}}

\NewDocumentCommand{\integ}{m m O{} D<>{}}{
\IfNoValueTF{#3}{\ensuremath{\displaystyle \int_{#2} #4 \, \mathrm{d}#1}}
{\ensuremath{\displaystyle \int_{#2}^{#3} #4 \, \mathrm{d}#1}}
}

\let\RealFrac\frac
\NewDocumentCommand\f{mg}{%
  \IfNoValueTF{#2}
    {\RealFrac{1}{#1}}
    {\RealFrac{#1}{#2}}%
}

\usepackage[normalem]{ulem}
\newcommand{\mathsout}[1]

	\title{
		{Existence and Uniqueness of Fast Traveling Pulses in Singularly Perturbed Nonlocal Neural Fields With Heaviside Nonlinearities: a Complete Proof}\\
	}
	\author{Alan Dyson\footnote{(AlanDyson.Research@gmail.com)}}
	\date{\today}
\allowdisplaybreaks
\begin{document}
\newcounter{count}
\setcounter{count}{1}
\newcommand{\Ecount}{\thecount\stepcounter{count}}

\maketitle
\abstract 
We rigorously prove the existence and uniqueness of fast traveling pulse solutions to the singularly perturbed neural field system with linear feedback and Heaviside nonlinearity structure within a spatial convolution. Although a long-standing open problem, the pulse is well-accepted to often exist based on its original singular construction, closed form when it exists, and follow-ups, but prior to this study, there has not been a proof that overcomes the difficulties of (i) solving for the fast speed and width functions using the implicit function theorem at $\epsilon=0$ and (ii) tracking the resultant formal homoclinic orbit near its singular orbit during fast, slow, and mixed time scales. First, we provide new, first-order approximations of the pulse speed and width. We then show that the formal pulse is close to the front and back at threshold crossing points and that the Hausdorff distance between the formal pulse and the singular homoclinic orbit converges to zero, demonstrating that the formal pulse is a true solution to the original system. Broadly, our methods provide insight into nonlocal geometric singular perturbation theory.

\section{Introduction}\label{Section: Intro}
In mammals, cortical waves are an intriguing and observable pattern, both in normal function and pathology. For example, wave propagation is a prominent phenomena seen in higher order sensory and motor cortical processing, believed to encode plasticity and other hierarchical sequencing \cite{sato2022cortical}. This occurs in the visual cortex via feedforward and feedback traveling waves \cite{aggarwal2022v}. By combining high-resolution imaging techniques like fMRI and MEG-EEG, a synergistic approach can provide insight into the connection between retinotopic directional mappings, hemodynamics, and vision-related motor tasks \cite{zanos2015sensorimotor,aquino2012hemodynamic,grabot2025traveling}.
Traveling waves have also been proven to play a pivotal role in cortical organization such as in working memory \cite{luo2025traveling}. In vitro, neural slices have also shown evidence of recorded propagation, often with techniques to block inhibition \cite{avoli2016models,pavan2022vitro}. Sensory processing pathologies such as migraines \cite{o2021migraine} and epilepsy \cite{martinet2017human} are also known to be related to pulse propagation.

Due to a tractability challenge in neuroscience, neural fields are a well-studied macroscopic model that portray the spatial and temporal interactions between patches of neurons as nonlocal integro-differential equations. Originally, the neural field equations were modeled by Wilson and Cowan \cite{WilsonCowan1972} with separate excitatory and inhibitory network populations. Amari \cite{Amari1977} justified a scalar equation of mixed populations with firing rates encoded within integrals of synaptic interactions. Models with one spatial dimension are justified, based on the layering structure of the cortex and practical recording methods. Voltage activity is then measured as spatial and temporal averages with spike times neglected. Efforts have been made to fit Wilson-Cowan and Amari models to data \cite{freestone2011data,rule2019neural,kofinas2023latent}.

With further simplification, Amari introduced the so-called ``Heaviside" activation approach that supplements the usual sigmoidal models with a more tractable alternative. By exploiting the Heaviside structure, Amari showed bumps, oscillations, and traveling waves exist when synaptic coupling kernels take on Mexican hat shapes. Since then, there have been many follow-ups based on Amari's Heaviside assumption.

As a scalar equation, the model takes on the form \cite{Amari1977,ExistenceandUniqueness-ErmMcLeod} 
\begin{equation} \label{eq:	 intro_scalar}
u_t=-u+\integ{y}{\R}[]<K(x-y)H(u(y,t)-\theta)>.
\end{equation}
In \eqref{eq: intro_scalar}, $u=u(x,t)$ represents average voltage at position $x$ and time $t$; $K$ is a synaptic coupling kernel, representing homogeneous spatial connectivity strength between presynaptic neural patches at position $y$ with postsynaptic neural patches at position $x$. The parameter $\theta>0$ is the threshold for neurons at $y$ to fire. We ignore previously studied spatial and feedback delays \cite{HuttZhang-TravelingWave,LijZhangSolo,Atay2004} for technical reasons, though there's no obvious reason why they would alter our approach. There are a variety of other modifications of \eqref{eq: intro_scalar} in the literature, but \eqref{eq: intro_scalar} can be categorized as the standard nonlocal Heaviside model without external currents, synaptic depression, stochastic noise, or other nonlinearities.

Due to a Heaviside firing rate, neural patches communicate in an all-or-none manner, justified statistically (and mathematically \cite{Burlakov_etal2025}) as an average of sigmoidals or the limit of steep sigmoidals. While sigmoidal firing rates are more physically appropriate, the Heaviside function provides a favorable mathematical trade-off.

Incorporating more realistic linear feedback such as spike frequency adaptation, Pinto and Ermentrout also studied traveling and standing waves in the singularly perturbed system \cite{pinto2001spatially,PintoandErmentrout-SpatiallyStructuredActivityinSynapticallyCoupledI.TravelingWaves}
\begin{align}
u_t&=-u-q+ \integ{y}{\R}[]<K(x-y)H(u(y,t)-\theta)>, \label{eq: intro_system_1}\\
q_t &=\epsilon (u-\gamma q). \label{eq: intro_system_2}
\end{align}
Here, $0<\epsilon \ll 1$, the function $q$ is a slow linear term, and $\gamma>0$ is a decay constant. Equation \eqref{eq: intro_scalar}, system \eqref{eq: intro_system_1}--\eqref{eq: intro_system_2}, and variations have been studied extensively, leading to vast dynamical systems findings. See \cite{cowan2014personal,ermentrout2010mathematical, bressloff2014waves_wholebook,coombes2014neural,cook2022neural} for background and more historical context.

With traveling waves as our focus, we introduce the traveling coordinate $z=x+c t$ and look for solutions to \eqref{eq: intro_system_1}--\eqref{eq: intro_system_2} of the form $\allowbreak (u(x,t),q(x,t))=(U(z),Q(z))$. Moreover, this solution is defined as a fast traveling pulse if there exists $a=O\left(\f{\epsilon}\right)$ such that $U(z)<\theta$ on $(-\infty,0)\cup (a,\infty)$ and $U(z)>\theta$ on $(0,a)$. Plugging in this ansatz into \eqref{eq: intro_system_1}--\eqref{eq: intro_system_2}, we arrive at the system
\begin{align}
c\begin{pmatrix} U' \\ Q' \end{pmatrix}
+A
\begin{pmatrix} U \\ Q \end{pmatrix}
= \begin{pmatrix} \integ{x}{z-a}[z]<K(x)> \\ 0 \end{pmatrix}, \label{eq: ansatz}
\end{align}
where $A:=\begin{pmatrix}
1 & 1\\
-\epsilon & \epsilon\gamma
\end{pmatrix}$. The matrix $A$ has eigenvalues and eigenvectors given by \begin{align}
\omega_1(\epsilon)&=\frac{1+\gamma\epsilon+\sqrt{(1-\gamma\epsilon)^2-4\epsilon}}{2},\qquad v_1=\begin{pmatrix}1\\\omega_1-1\end{pmatrix}, \\
\omega_2(\epsilon)&=\frac{1+\gamma\epsilon-\sqrt{(1-\gamma\epsilon)^2-4\epsilon}}{2},\qquad v_2=\begin{pmatrix}1\\\omega_2-1\end{pmatrix}.
\end{align}
Under the assumption of fixed $\gamma>0$ and $\epsilon \ll 1$, both $\omega_1$ and $\omega_2$ are positive with
\begin{equation}\label{eq: limits_eig}
\omega_1(0)=1,\qquad \omega_2(0)=0,\qquad \omega_1'(0)=-1, \qquad \omega_2'(0)=1+\gamma.
\end{equation}
Note that 
\begin{equation}\label{eq: omega_calcs}
\omega_1 \omega_2 =\det(A)=\epsilon(1 + \gamma), \qquad \omega_1 + \omega_2 = tr(A) = 1+\epsilon \gamma,
\end{equation}
and from \eqref{eq: omega_calcs}, we see that 
\begin{equation}
\f{1-\omega_1}{\omega_2}=\f{\omega_2-\epsilon\gamma}{\omega_2}=1-\f{\epsilon\gamma}{\omega_2}=1-\f{\omega_1\gamma}{1+\gamma}. \label{eq: OneMinusOmega1_Omega2}
\end{equation}
Furthermore, by differentiating the equations in \eqref{eq: omega_calcs} twice, we find $\omega_1''\omega_2 +2\omega_1'\omega_2' +\omega_1\omega_2''=0$ and $\omega_1''+\omega_2''=0$ so
\begin{equation}
\omega_1''(0)=-2(1+\gamma),\qquad \omega_2''(0)=2(1+\gamma). \label{eq: omega_secder}
\end{equation}
The formal solution to \eqref{eq: ansatz} is given by \cite{PintoandErmentrout-SpatiallyStructuredActivityinSynapticallyCoupledI.TravelingWaves,Pinto2005}
\begin{align}
U(z)&=\integ{x}{-\infty}[z]<C_x(x-z,c,\epsilon)\left(\integ{y}{x-a}[x]<K(y)>\right)>,\label{eq: U_formal} \\
Q(z)&=\integ{x}{-\infty}[z]<D_x(x-z,c,\epsilon)\left(\integ{y}{x-a}[x]<K(y)>\right)>, \label{eq: Q_formal}\\
U'(z)&=\integ{x}{-\infty}[z]<C_x(x-z,c,\epsilon)\left(K(x)-K(x-a)\right)>, \label{eq: Up_formal}\\
Q'(z)&=\integ{x}{-\infty}[z]<D_x(x-z,c,\epsilon)\left(K(x)-K(x-a)\right)>,\label{eq: Qp_formal}
\end{align}
where 
\begin{align}
C(x,c,\epsilon)&=\frac{1}{\omega_1-\omega_2}\left[\f{1-\omega_2}{\omega_1}e^{\frac{\omega_1 x}{c}}-\f{1-\omega_1}{\omega_2}e^{\frac{\omega_2 x}{c}}\right], \label{eq: C}\\
C_x(x,c,\epsilon)&=\frac{1}{c(\omega_1-\omega_2)}\left[(1-\omega_2)e^{\frac{\omega_1 x}{c}}-(1-\omega_1)e^{\frac{\omega_2 x}{c}}\right], \label{eq: Cx}\\
D(x,c,\epsilon)&=\frac{\epsilon}{\omega_1-\omega_2}\left[-\f{\omega_1}e^{\frac{\omega_1 x}{c}}+\f{\omega_2}e^{\frac{\omega_2 x}{c}}\right], \label{eq: D} \\
D_x(x,c,\epsilon)&=\frac{\epsilon}{c(\omega_1-\omega_2)}\left[-e^{\frac{\omega_1 x}{c}}+e^{\frac{\omega_2 x}{c}}\right]. \label{eq: Dx}
\end{align}
Using the properties of $\omega_1$ and $\omega_2$ from \eqref{eq: limits_eig}--\eqref{eq: omega_calcs}, we see that
\begin{align*}
C(0,c,\epsilon)&=\f{\omega_2(1-\omega_2)-\omega_1(1-\omega_1)}{\omega_1\omega_2(\omega_1-\omega_2)}=\f{-1+\omega_1+\omega_2}{\omega_1\omega_2}=\f{\gamma}{1+\gamma}, \numberthis \label{eq: C0}\\
D(0,c,\epsilon)&=\f{\epsilon}{\omega_1\omega_2}=\f{1+\gamma}. \numberthis \label{eq: D0}
\end{align*}

Clearly, $(U(-\infty),Q(-\infty))=(0,0)$, and by splitting the integrals by $\int_{-\infty}^z=\int_{-\infty}^a+\int_a^z$, an easy application of the dominated convergence theorem shows $(U(+\infty),Q(+\infty))=(0,0)$ as well. Hence, even for arbitrary $a,c>0$, the formal solution can be understood as a homoclinic orbit connecting the point $(0,0)$ to itself.

This brings us to the main difficulty and purpose of this paper. When is a solution to \eqref{eq: ansatz} really a solution to \eqref{eq: intro_system_1}--\eqref{eq: intro_system_2}? Assuming only the existence of a front (that we show is always unique), how do we prove the existence of unique parameters $(a(\epsilon),c(\epsilon))$ that solve the compatibility equations $U(0)=U(a)=\theta$? How do we demonstrate that the sub and super threshold regions truly correspond and that the pulse is close to the singular homoclinic orbit? Can we compare the difference between the pulse and front speed up to first order in $\epsilon$? We answer all of these questions.

\subsection{Main Results}\label{subsec: main_goal}
Our main goal is to rigorously prove the existence and uniqueness of fast traveling pulses to system \eqref{eq: intro_system_1}-\eqref{eq: intro_system_2} under minimal assumptions. In particular, we prove the following theorems, assuming hypotheses (H1)--(H3) discussed in \cref{subsec: hyp}.
\begin{theorem}[Existence and Uniqueness of Fast Pulses]\label{thm: E}
When $0<\epsilon \ll 1,$ there exists a unique $C^2$ (modulo translation) fast traveling pulse solution $(U_\epsilon,Q_\epsilon)$ to system \eqref{eq: intro_system_1}-\eqref{eq: intro_system_2}. In particular, with $z=x+ct$, there exists unique $C^1$ curves $c_\epsilon=c(\epsilon)$ and $a_\epsilon=a(\epsilon)$ such that $U_\epsilon(z)<\theta$ on $(-\infty,0)\cup(a_\epsilon,\infty)$ and $U_\epsilon(z)>\theta$ on $(0,a_\epsilon)$. The parameters satisfy $c_\epsilon\to c_f$ and $a_\epsilon\to +\infty$ as $\epsilon \to 0$, where $c_f>0$ is the unique speed of the front solution to \eqref{eq: intro_scalar}. As a homoclinic orbit, $\mathcal{S}_\epsilon=\Bigl\{(U_\epsilon(z),Q_\epsilon(z)): z\in \R \Bigr\}$ satisfies $d_H(\mathcal{S}_\epsilon,\mathcal{S}_0)\to 0$ as $\epsilon\to 0$. Here, 
$$
d_H(\mathcal{S}_\epsilon,\mathcal{S}_0)=\max\left\lbrace \sup_{(U,Q)\in \mathcal{S}_0}d\Bigl(\mathcal{S}_\epsilon,(U,Q)\Bigr),\sup_{(U_\epsilon,Q_\epsilon)\in \mathcal{S}_\epsilon}d\Bigl((U_\epsilon,Q_\epsilon),\mathcal{S}_0\Bigr) \right\rbrace
$$ 
is the Hausdorff distance between sets in $(U,Q)$ space with respect to the standard Euclidean metric, and $\mathcal{S}_0$ is the singular homoclinic orbit at $\epsilon=0$.
\end{theorem}
As a result of the original work of Pinto and Ermentrout \cite{PintoandErmentrout-SpatiallyStructuredActivityinSynapticallyCoupledI.TravelingWaves}, as well as follow-up mathematical analysis (such as  by Pinto et al. \cite{Pinto2005}), there is a heuristically and computationally justified acceptance of much of \cref{thm: E} (especially for nonnegative kernels), but a complete proof of existence and uniqueness is lacking prior to this paper. 

In particular, to the author's knowledge, there have been no other results that rigorously study the fast pulse solutions in the Heaviside model when they are broken down by fast and slow representations in inner and outer regions for $\epsilon \ll 1$. Unlike the singular homoclinic orbit, when $\epsilon>0$, the fast and slow dynamics need to be analyzed when they overlap. When $K$ is chosen as the exponential type. we have PDE reductions and applications of classic geometric singular perturbation theory can be applied. We avoid requiring this reduction.

In order to prove \cref{thm: E}, we derive a number of new estimates of the solutions and their wave speeds and widths. The following corollary summarizes first-order approximations of the parameters. Denote $\tau(\epsilon):=\epsilon a(\epsilon)$, a natural change of variable since $a(\epsilon)=O\left(\f{\epsilon}\right)$ is known based on the fast and slow time scales. The functions $f$ and $g$ are discussed at length in \cref{sec: calc}.
\begin{corollary}[First-Order Approximations of $\tau(\epsilon)$ and $c(\epsilon)$] \label{thm: first_order}Consider the unique pair of $C^1$ curves $\tau_\epsilon=\tau(\epsilon)$ and $c_\epsilon=c(\epsilon)$ from \cref{thm: E}. The limits $\tau(\epsilon)\to \tau_0$ and $c(\epsilon)\to c_f$ hold, where
$$
\tau_0=-\f{c_f}{1+\gamma}\ln(2\theta(1+\gamma)-\gamma).
$$
The derivative of the vector $(\tau(\epsilon),c(\epsilon))^T$ satisfies
\begin{equation}
 \begin{pmatrix}
\tau'(0) \\ c'(0)
\end{pmatrix}
=-J^{-1}(\tau_0,c_f,0)
\begin{pmatrix}
f_\epsilon(\tau_0,c_f,0) \\ g_\epsilon(\tau_0,c_f,0)
\end{pmatrix},
\end{equation}
where $f(\tau(\epsilon),c(\epsilon),\epsilon)=U(0)=\theta$ and $g(\tau(\epsilon),c(\epsilon),\epsilon)=U(a(\epsilon))=\theta$ are the speed index equations for the pulse. Hence, $\tau(\epsilon)\approx \tau_0 +\tau'(0)\epsilon$ and $c(\epsilon)\approx c_f +c'(0)\epsilon$ are first-order approximations of $\tau(\epsilon)$ and $c(\epsilon)$, respectively. In particular,
\begin{equation}\label{eq: cp_zero}
c'(0)=-\f{\phi_f'(c_f)}\left\lbrace 2\theta-\f{c_f}\integ{x}{-\infty}[0]<|x|\left(e^{\f{x}{c_f}}+1\right)K(x)>\right\rbrace,
\end{equation}
where $\phi_f$ is the speed index function for the front that satisfies $\phi_f'(c_f)<0$.
\end{corollary}
\subsection{Hypotheses} \label{subsec: hyp}
We state our main hypotheses and then offer some remarks as to why they are necessary.
\begin{itemize}
\item[(H1)] The parameters satisfy $0<\theta<\min\{\f{2},\int_{-\infty}^0 K\}$, \, $\f{\gamma}{1+\gamma}<\theta,\,$ and $\epsilon \ll 1$.
\item[(H2)] When $\epsilon=0$, there exists a traveling front solution $u(x,t)=U_f(x+c_ft)$ to \eqref{eq: intro_scalar} with $U_f(\cdot)<\theta$ on $(-\infty,0)$, $U_f(\cdot)>\theta$ on $(0,\infty)$, $U_f(-\infty)=0$, and $U_f(\infty)=1$. The solution is translation invariant, but we will assume $U_f(0)=\theta$ unless stated otherwise.
\item[(H3)] There exists $\alpha>0$ and $\rho>0$ such that $|K(x)|\leq \alpha e^{-\rho|x|}$. Moreover, assume $K$ is continuous and $\int_\R K = 1$.
\end{itemize}
We require $0<\theta<\int_{-\infty}^0 K$ because it is necessary in order for the front to exist, while $\theta <\f{2}$ is required for the pulse. Often $K$ is symmetrical so these conditions are equivalent. To conclude remarks on (H1), we note that the assumption $\frac{\gamma}{1+\gamma}<\theta$ ensures that the line $Q=\frac{U}{\gamma}$ does not intersect $Q=-U+H(U-\theta)$ when $U>\theta$, which would lead to a second fixed point, and the homoclinic orbit would be replaced by a heteroclinic orbit. Hypothesis (H2) ensures we have a leading order inner solution. Finally, hypothesis (H3) is a common assumption in nonlocal problems. Biologically, long-range synaptic interaction strengths decay sharply, while mathematically, exponential decay ensures that terms of order $O\left(\epsilon^{-k}e^{-\f{|C|}{\epsilon}}\right)$ vanish. 
\section{Singular Solution,  Previous Results, and Outline}
We review the Heaviside version of the $\epsilon=0$ singular homoclinic orbit construction from \cite{PintoandErmentrout-SpatiallyStructuredActivityinSynapticallyCoupledI.TravelingWaves}, which proves the leading orders of inner and outer solutions can be matched at boundary points. We then highlight how our result advances nonlocal geometric singular perturbation theory approaches.
\subsection{Construction of Singular Homoclinic Orbit}
\subsubsection*{Fast Inner Solutions at $\epsilon=0$}
Writing \eqref{eq: intro_system_1}--\eqref{eq: intro_system_2}, in terms of the traveling coordinate, we obtain the fast inner system
\begin{align}
c\begin{pmatrix} U' \\ Q' \end{pmatrix}
+\begin{pmatrix}
1 & 1\\
-\epsilon & \epsilon\gamma
\end{pmatrix}
\begin{pmatrix} U \\ Q \end{pmatrix}
= \begin{pmatrix} \integ{y}{\R}[]<K(z-y)H(U(y)-\theta)> \\ 0 \end{pmatrix}. \label{eq: fast_null}
\end{align}
When $\epsilon=0$, from the second equation, we have $Q'=0$ so $Q\equiv Q_0$ for some constant $Q_0$. Plugging into the first equation, we recall from \cite{ExistenceandUniqueness-ErmMcLeod} that there may (guaranteed when $K\geq 0$) exist a front with wave speed $c_f>0$ that solves
\begin{equation}
c_f U_f' + U_f=\integ{y}{\R}[]<K(z-y)H(U_f(y)-\theta)>.
\label{eq: front_ODE}
\end{equation}
Here, $Q_0=0$ and the solution is given by
\begin{align}
U_f(z)&=\frac{1}{c_f}\integ{x}{-\infty}[z]<e^{\frac{x-z}{c_f}}\left(\integ{y}{-\infty}[x]<K(y)>\right)>, \label{eq: front}\\
U_f'(z)&=\frac{1}{c_f}\integ{x}{-\infty}[z]<e^{\frac{x-z}{c_f}}K(x)>. \label{eq: frontp}
\end{align}
The limits $U_f(-\infty)=0,\, U_f(\infty)=1,\, U_f'(\pm \infty)=0$ hold, and the solution satisfies $U_f(z)<\theta$ on $(-\infty,0)$ and $U_f(z)>\theta$ on $(0,\infty)$. The front speed $c_f$ is calculated as a root of of the so--called speed index function $\phi_f(c)=U_f(0)=\theta$. Explicitly,
\begin{align*}
\phi_f(c)&:=\frac{1}{c}\integ{x}{-\infty}[0]<e^{\frac{x}{c}}\left(\integ{y}{-\infty}[x]<K(y)>\right)>=\integ{x}{-\infty}[0]<K(x)>-\integ{x}{-\infty}[0]<e^{\f{x}{c}}K(x)>,\numberthis \label{eq: phi_def}\\
\phi_f'(c)&=-\f{c}\phi_f(c)-\f{c^3}\integ{x}{-\infty}[0]<xe^{\frac{x}{c}}\left(\integ{y}{-\infty}[x]<K(y)>\right)>=\f{c^2}\integ{x}{-\infty}[0]<xe^{\f{x}{c}}K(x)>. \numberthis \label{eq: phip_def}
\end{align*}
In the fast system, the inner solution approximates the front as it completes most of its trajectory from $(U_\epsilon(-\infty),Q_\epsilon(-\infty))=(0,0)$ to $(U_\epsilon,Q_\epsilon)\approx (1,0)$ in $O(1)$ time. Then the slow outer solution takes over during upward travel near the right critical manifold (discussed below), and it takes $O\left(\f{\epsilon}\right)$ time until $(U_\epsilon(z),Q_\epsilon(z))\approx (2\theta,1-2\theta)$.

It can then be shown that when $Q_0=1-2\theta,$ the leading order of the fast inner solution generates a traveling back $U_b:=2\theta-U_f$ with the same wave speed (i.e. $c_b=c_f$). The sub and super threshold regions are then opposite of those for the front. The solution approximates the back in $O(1)$ time again, preparing the slow left critical manifold (discussed below) to carry the solution from $(U_\epsilon,Q_\epsilon)\approx (2\theta-1,1-2\theta)$ to $(0,0)$ in $O\left(\f{\epsilon}\right)$ time.

Combined, the fast equations provide two trajectories in $(U,Q)$ space: the front and the back, which we write as
\begin{align}
\mathcal{F}_0:=\{(U_f(z),0): z\in \R\}, \qquad \mathcal{B}_0:=\{(U_b(z),1-2\theta): z\in \R\}.\label{eq Front_Back_sing}
\end{align}
\subsubsection*{Slow Outer Solutions at $\epsilon=0$}
Now consider slow time with the coordinate $\tau=\epsilon z$. Then we obtain the system
\begin{align}
c\begin{pmatrix} \epsilon \dot{U} \\ \dot{Q} \end{pmatrix}
+\begin{pmatrix}
1 & 1\\
-1 & \gamma
\end{pmatrix}
\begin{pmatrix} U \\ Q \end{pmatrix}
= \begin{pmatrix} \integ{y'}{\R}[]<\f{\epsilon}K\left(\frac{\tau-y'}{\epsilon}\right)H(U(y')-\theta)> \\ 0 \end{pmatrix}. \label{eq: slow_null}
\end{align}
Originally from \cite{PintoandErmentrout-SpatiallyStructuredActivityinSynapticallyCoupledI.TravelingWaves},  the firing rate was smooth so  $\{K_\epsilon\}_{\epsilon>0}$ was treated as a family of functions that converge to the delta distribution. In our case, the Heaviside function has a discontinuity, but causes no issues since our left (resp. right) slow critical manifolds are only of interest when $U(\tau)<\theta$ (resp. $U(\tau)>\theta$). Therefore, as $\epsilon\to 0$, we obtain the leading order of the slow system
\begin{align}
Q&=-U+H(U(\tau)-\theta), \\
c\dot{Q}&=U-\gamma Q.
\end{align}
Hence, we may define the left and right critical manifolds in neighborhoods of interest by
\begin{align}
\widehat{\mathcal{M}}_0^L:=\{(-Q,Q): Q\in[-\delta_0,(1-2\theta)+\delta_0]\}, \quad \widehat{\mathcal{M}}_0^R:=\{(1-Q,Q): Q\in [-\delta_0,(1-2\theta)+\delta_0]\}. \label{eq: slow_crit_manifolds}
\end{align}
By matching boundary conditions at takeoff and landing points of the front and back, we obtain the singular homoclinic orbit, defined by
\begin{equation} \label{eq: S_0_def}
\mathcal{S}_0:=\mathcal{F}_0\cup \mathcal{M}_0^R\cup \mathcal{B}_0\cup \mathcal{M}_0^L.
\end{equation}
Here, $\mathcal{M}_0^L$ and $\mathcal{M}_0^R$ restrict to $Q\in [0,1-2\theta]$. 

By the assumptions on the parameters, we can readily see that $\dot{Q}<0$ (resp. $\dot{Q}>0$) on $\mathcal{M}_0^L$ (resp. $\mathcal{M}_0^R$). The true pulse is then a homoclinic orbit that is close to the singular homoclinic orbit. We use the Hausdorff distance metric between sets to show we can control the error completely as $\epsilon \to 0$. See Figure \ref{fig: sing_kt}.
\begin{figure}[H]
\centering
\includegraphics[width=100mm]{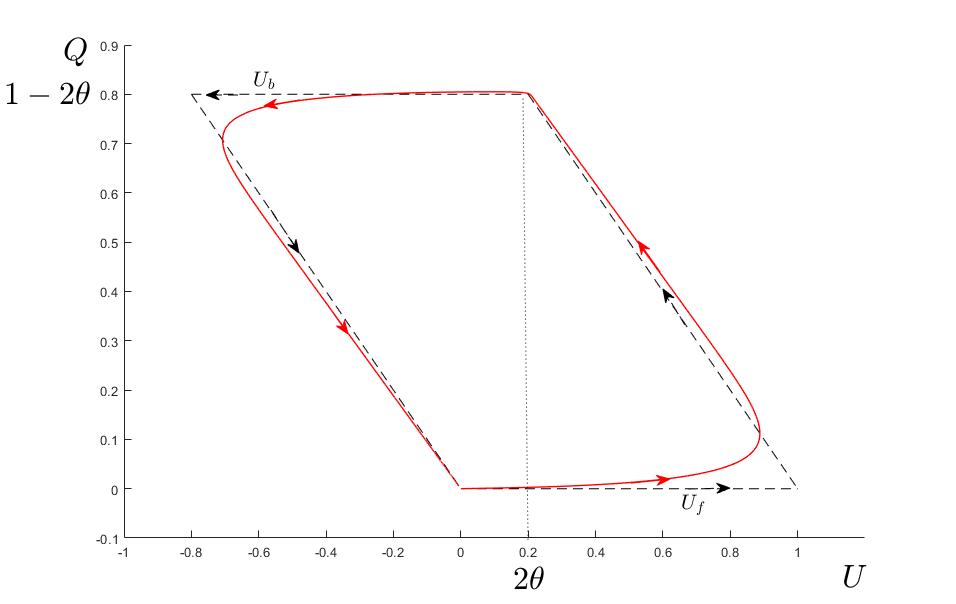}
\caption{Generalized phase space of $\mathcal{S}_0$ (black) along with $\mathcal{S}_\epsilon$ (red) for $\epsilon\ll 1$}
\label{fig: sing_kt}
\end{figure}

Note that if $U_f$ is non-monotone, the front and back repeat points in phase space when $Q=Q_0$. Moreover, it is possible for $U_f<0$ (resp. $U_b>2\theta$) or $U_f>1$ (resp. $U_b<2\theta -1$). For illustrative purposes only, we show the monotone case.
\subsection{Comparison to Previous Results}
Over the course of decades, there have been many results related to the Heaviside model, let alone variations and other nonlocal problems. In this subsection, we highlight some foundational studies that are directly relevant to our main results.
\subsubsection*{Foundational Work}
In the scalar equation, the front is a necessary area of study by itself. For sigmoidal firing rates, the existence (and local uniqueness) of fronts was first proven in the setting of nonnegative kernels by Ermentrout and McLeod in their pioneering work \cite{ExistenceandUniqueness-ErmMcLeod}. In nonlocal equations covering neural fields with nonnegative kernels, Chen \cite{Chen1997} provided strong proofs of uniqueness and asymptotic stability.

With fronts established in the Heaviside and sigmoidal cases, as discussed above, Pinto and Ermentrout \cite{PintoandErmentrout-SpatiallyStructuredActivityinSynapticallyCoupledI.TravelingWaves}
established the construction of two traveling pulses in the singularly perturbed system, one fast and one slow. The fast pulse is close to the singular homoclinic orbit. The slow pulse will not be focused on here since it is unstable. By highlighting the different time scales, they established the singular solution and computed the fast pulse without rigorous proof of inner and outer matching when $\epsilon>0$. Their use of geometric singular perturbation theory was only applied in the sigmoidal case. For the Heaviside case, they suggested methods for solving for $(a,c)$ with fixed $\epsilon$, but did not study sub and super threshold regions.

Building off of \cite{PintoandErmentrout-SpatiallyStructuredActivityinSynapticallyCoupledI.TravelingWaves}, substantial progress towards a proof was given by Pinto, Jackson, and Wayne \cite{Pinto2005}. They showed that for positive ``bell-shaped" kernels, the fast (and slow) parameter pair $(a,c)$ can be solved for when $\theta$ is small. However, the calculations are directly based on that kernel shape. We still note that our intuition behind solving $U(0)=U(a)=\theta$ for $(a,c)$ in this paper is motivated by some of their findings.

Our results depart entirely from \cite{Pinto2005} in several key ways. Rather than fixing $\epsilon$, we use the time scale $\tau=\epsilon a$ in order to draw a comparison between the front and pulse speed index functions as $\epsilon\to 0$, with the kernel shape being arbitrary when we apply the implicit function theorem. Indeed, our proofs of \cref{thm: E,thm: first_order} cover pulses with highly non-monotone fronts and backs. The existence and uniqueness of such traveling and stationary wave solutions to the scalar model have been previously studied \cite{Zhang-HowDo,Dyson2019_MBE,Lvwang,magpantay2010wave,ExploitingtheHamiltonian}.
See Figure \ref{fig: pulse_osc}, for example.

After obtaining the $C^1$ curves $(a(\epsilon),c(\epsilon))$, another key difference is that we carefully study the time scales and show formal pulse solutions are actual solutions to \eqref{eq: intro_system_1}--\eqref{eq: intro_system_2} and are arbitrarily close to the singular homoclinic orbit. The authors in \cite{Pinto2005} did not track sub and super threshold regions beyond computationally checking, and remarked that for some kernel choices, formal solutions fail the threshold requirements. We prove that exponential decay of the kernel is a sufficient hypothesis to avoid this situation.
\begin{figure}[H]
\centering
\includegraphics[width=100mm]{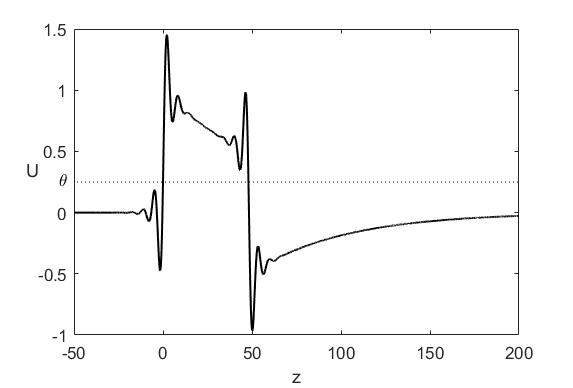}
\caption{A traveling pulse solution when $\theta=0.25$, $\gamma=0.2$, $\epsilon=0.005$, and $K(x)=\f{1+a^2}{4a}e^{-a|x|}(a\sin(|x|)+\cos(x))$ with $a=0.3$. The front and back are non-monotone, but cross the threshold only once}
\label{fig: pulse_osc}
\end{figure}
\subsubsection*{Geometric Singular Perturbation Theory and Other Methods}
In singularly perturbed systems, tracking the trajectory of the pulse can be exceptionally difficult. We provide a brief overview and comparison to common techniques in geometric singular perturbation theory.

Importantly, Fenichel  \cite{fenichel1979geometric} showed that in local and autonomous systems, the $\epsilon=0$ invariant critical manifolds (and their stable and unstable manifolds) can be perturbed. However, the pulse problem is still challenging because the shooting unstable manifold (the homoclinic orbit) may be of lower dimension than the stable and unstable manifolds of the critical manifolds. In the context of normally hyperbolic critical manifolds, Jones and Kopell \cite{JonesKopell} used differential forms to create the $k+1$ Exchange Lemma in order to compare manifolds of different dimensions, demonstrating that transverse intersections of various stable and unstable manifolds at $\epsilon=0$ is a sufficient condition in order to pick the wave speed. Jones et al. \cite{jones37construction} successfully constructed the Fitzhugh-Nagumo pulse with this method. The Exchange Lemma with exponentially small error (ELESE) was later developed to solve the multiple-pulse problem in more general systems like the Hamiltonian \cite{jones1996tracking}.

In neural fields with synaptic depression and sigmoidal firing rates, Faye \cite{Faye2013} also applied geometric singular perturbation theory to study existence and stability. However, the author required the kernel to be an exponential type, forcing the problem to ultimately reduce to a system of local, first-order ODEs. In this case, the variable $q$ is multiplied inside the convolution rather than being a linear term so the take-off point for the back occurs at a knee, leading to a center manifold; a ``blow-up" method was then required. In general, kernels with special Fourier transforms (such as differences of exponentials \cite{guo2016existence}) lead to local structures of varying complexity.

Using a true nonlocal method, a powerful existence result was obtained by Faye and Scheel \cite{Faye2015}. For the singularly perturbed nonlocal Fitzhugh--Nagumo equations, they proved the existence of fast pulses. They invoked partitions of unity, Fredholm operator theory, fixed point theory, and other methods to show that the real pulse is essentially the singular pulse with remainder terms that converge to zero as $\epsilon \to 0$. Due to nonlinearity structures, their methods are more advanced than those in this paper, but there are parallels in approach that we adopt such as exploiting exponential decay of the kernel.

In the present study, the Heaviside firing rate leads to closed form solutions, so while the analysis in our study is thorough, the mathematical tools are simpler. We are able to directly show $d_H(\mathcal{S}_\epsilon,\mathcal{S}_0) \to 0$ by sufficiently estimating solutions at regions in the domain where inner and outer regions overlap.

\subsubsection*{Smoothness and Other Considerations}
Mathematically, it is important to understand the Heaviside function as a limit of steep sigmiodals. Burlakov, Oleynik, and Ponosov \cite{Burlakov_etal2025} showed that if traveling pulses exist in neural fields systems with smoothed Heaviside firing rates and nonnegative kernels, then we may increase the steepness of the firing rates as they converge to the Heaviside case, and the pulse solutions converge to the Heaviside solution. However, they do not attempt to analyze the formal Heaviside solution as we do here. Dyson \cite{Dyson2020} explored a similar idea involving the front in the presence of lateral inhibition couplings, but showed existence persists as the firing rate morphs from the Heaviside function to smoothed Heaviside functions. Since smoothed Heaviside functions can take on sigmoidal shapes, they serve as a bridge between the traveling wave problem with and without closed form solutions. Hence, some studies examined iterative computational techniques for approximating solutions \cite{CoombesSchmidt,oleynik2015iterative}.  

Beyond those mentioned here, there have been many other works concerning the pulse in neural fields, exploring concepts like stability \cite{Zhang-OnStability,Sandstede-EvansFunctions,CoombesOwen_Evans}, heterogeneities \cite{kilpatrick2008traveling}, breathers \cite{folias2017}, stochasticity \cite{BressloffWebber,lang2016multiscale}, and stimulus-driven wave shifts \cite{shaw2024representing}, to name a few. 

\subsection{Outline of Technical Analysis}
\cref{thm: E,thm: first_order} are proven with two main technical steps, as follows:
\begin{enumerate}
\item In \cref{sec: calc}, we show that for $\epsilon \ll 1$, there exists unique solution curves $(\tau(\epsilon), c(\epsilon))$. We accomplish this using the implicit function theorem at $\epsilon=0$. We set up a mapping $F(\tau,c,\epsilon)$, show it is $C^1$, solve for the only eligible root, and prove $\det J \neq 0$. These calculations allow us to prove \cref{thm: first_order}.
\item In \cref{sec: thres}, we prove that $U_\epsilon(z)<\theta$ on $(-\infty,0)\cup (a(\epsilon),\infty)$ and $U_\epsilon(z)>\theta$ on $(0,a(\epsilon))$. We accomplish this by creating a sliding partition of $\R$. For all $\epsilon \ll 1$, the partition consists of four sets (for the front, back, left, and right). By rewriting the formal solution $(U_\epsilon(z),Q_\epsilon(z))$ in appropriate ways in each region, we show the solution is arbitrarily close to the singular homoclinic orbit.
\end{enumerate}
\section{Calculation of Unique $(a_\epsilon,c_\epsilon)$}\label{sec: calc}
In this section, we apply the implicit function theorem in order to prove the existence and uniqueness of $(a(\epsilon),c(\epsilon))$ for all $\epsilon<\epsilon_p$ for some $\epsilon_p>0$. We start with a proof that (if it exists) the front is always unique. We then study the pulse as a system.
\subsection{Uniqueness of the Front and Back Speed}\label{subsec: Unique_front}
In this subsection, we prove the following lemma.
\begin{lemma}\label{lemma: unique_front}
Suppose $K$ and $\theta$ are chosen so that a traveling front solution $(U_f,c_f)$ exists to \eqref{eq: front_ODE}. Then $c_f$ is the only positive root of the equation $\phi_f(c)=\theta$. Hence, the solution $(U_f,c_f)$ is unique (modulo translation). Moreover, $\phi_f'(c_f)<0$ holds.
\end{lemma}
\begin{proof}
The proof is trivial when $K\geq 0$ since $\phi_f(0)=\int_{-\infty}^0 K>\theta$, $\phi_f(\infty)=0<\theta$, and $\phi_f'(c)<0$ for all $c>0$. But since our hypotheses leave open conditions on $K$, we prove a general claim (covering lateral inhibition, lateral excitation, oscillatory types).

For $\lambda >-1$, consider the function $\lambda \mapsto \phi_f\left(\f{c_f}{\lambda+1}\right)$. We show that $\lambda=0$ is the only solution to $\phi_f\left(\f{c_f}{\lambda+1}\right)=\theta$. From \eqref{eq: phi_def}, write
\begin{align*}
\phi_f\left(\f{c_f}{\lambda+1}\right)&=\frac{\lambda+1}{c_f}\integ{x}{-\infty}[0]<e^{\frac{(\lambda+1)x}{c_f}}\left(\integ{y}{-\infty}[x]<K(y)>\right)>\\
&=(\lambda+1)\int_{-\infty}^0 e^{\f{\lambda x}{c_f}}\left[\f{c_f} \int_{-\infty}^x e^{\f{y}{c_f}}\left(\integ{z}{-\infty}[y]<K(z)>\right)\,\mathrm{d}y \right]'\, \mathrm{d}x \\
&=(\lambda+1)\left\lbrace \phi_f(c_f)-\f{\lambda}{c_f}\int_{-\infty}^0 e^{\f{(\lambda+1) x}{c_f}}\left[\f{c_f} \int_{-\infty}^x e^{\f{y-x}{c_f}}\left(\integ{z}{-\infty}[y]<K(z)>\right)\,\mathrm{d}y \right]\, \mathrm{d}x\right\rbrace \\
&=(\lambda+1)\left\lbrace \theta-\f{\lambda}{c_f}\integ{x}{-\infty}[0]<e^{\f{(\lambda+1)x}{c_f}}U_f(x)> \right\rbrace.
\end{align*}
Since $U_f(x)<\theta$ for $x<0$, it follows that when $\lambda>0$ (resp. $\lambda<0$), we have $\phi_f\left(\f{c_f}{\lambda+1}\right)>\theta$ (resp. $\phi_f\left(\f{c_f}{\lambda+1}\right)<\theta$).
Hence, the only solution occurs at $\lambda=0$. As a result, the front has unique wave speed so is unique (modulo translation).

For the second claim, from \eqref{eq: phip_def}, write
\begin{align*}
\phi_f'(c_f)&=\f{c_f}\integ{x}{-\infty}[0]<x\left[\f{c_f}\integ{y}{-\infty}[x]<e^{\f{y}{c_f}}K(y)>\right]'>=-\f{c_f}\integ{x}{-\infty}[0]<e^{\f{x}{c_f}}U_f'(x)> \\
&=-\f{c_f}\left[\theta-\f{c_f}\integ{x}{-\infty}[0]<e^{\f{x}{c_f}}U_f(x)>\right].
\end{align*}
Again, since $U_f(x)<\theta$, it follows that $\f{c_f^2}e^{\f{x}{c_f}}U_f(x)<\f{\theta}{c_f^2}e^{\f{x}{c_f}}$ so $\phi_f'(c_f)<0$. This completes the proof.
\end{proof}
\begin{remark}
Concerning spectral stability, it can be shown for the front that roots of the Evans function \cite{Zhang-OnStability} satisfy $\mathcal{E}_f(\lambda)=0$, where 
$\mathcal{E}_f(\lambda):=\phi_f\left(\f{c_f}{\lambda+1}\right)-\theta$. These roots classify the point spectrum. By \cref{lemma: unique_front}, it follows that $\mathcal{E}_f(\lambda) > 0$ for real $\lambda>0$ and $\mathcal{E}_f'(0)=-c_f\phi_f'(c_f)>0$, showing that $\lambda=0$ is a simple eigenvalue. Thus, the only threats to stability occur when $\mathrm{Re}(\lambda)\geq 0$ and $\mathrm{Im}(\lambda)\neq 0$. Usually, the front and pulse are stable.
\end{remark}
\subsection{Preparation for the Implicit Function Theorem}\label{subsec: par}
We prepare to apply the implicit function theorem.
Intuitively, $U(a)$ is written in terms of the fast coordinate, but $U(a)=\theta$ occurs after an $O\left(\f{\epsilon}\right)$ outer region. Hence, as a form of time-scale matching, we set $\tau=\epsilon a$. Then from the two equations $U(0)=U\left(\f{\tau}{\epsilon}\right)=\theta$ in \eqref{eq: U_formal}, we arrive at the equations $f(\tau,c,\epsilon)=g(\tau,c,\epsilon)=\theta$, where
\begin{align}
f(\tau,c,\epsilon)&:=\integ{x}{-\infty}[0]<C_x(x,c,\epsilon)\left(\integ{y}{x-\f{\tau}{\epsilon}}[x]<K(y)>\right)>, \label{eq: f}\\
g(\tau,c,\epsilon)&:=\integ{x}{-\infty}[\f{\tau}{\epsilon}]<C_x\left(x-\f{\tau}{\epsilon},c,\epsilon\right)\left(\integ{y}{x-\f{\tau}{\epsilon}}[x]<K(y)>\right)>.\label{eq: g}
\end{align}
Define the vector--valued function $F:[\tau_0-\delta_\tau,\tau_0+\delta_\tau]\times [c_f-\delta_c,c_f+\delta_c]\times [0,\delta_\epsilon] \to \R^2$ by
\begin{equation}\label{eq: Fvec}
F(\tau,c,\epsilon):=
\begin{pmatrix}
f(\tau,c,\epsilon) \\
g(\tau,c,\epsilon)
\end{pmatrix}.
\end{equation}
Here, $\tau_0>0$ is derived below in \cref{subsec: IFT}. The positive numbers $\delta_\tau$, $\delta_c$, and $\delta_\epsilon$ are arbitrary and chosen so that $F$ is defined in a (positive) neighborhood of $(\tau_0,c_0,\epsilon_0)=(\tau_0,c_f,0)$, the base point where we apply the implicit function theorem. The strategy is to show that $F$ is $C^1$ and that $\det J \neq 0$, where
\begin{equation}\label{eq: Jacobian}
J=\begin{pmatrix}
f_\tau(\tau_0,c_0,\epsilon_0) & f_c(\tau_0,c_0,\epsilon_0)\\
g_\tau(\tau_0,c_0,\epsilon_0) & g_c(\tau_0,c_0,\epsilon_0)
\end{pmatrix}.
\end{equation}
\subsubsection{The Continuous Extension of $F$}\label{subsubsec: F_C1}
In this subsection, we show that $F$ can be continuously extended when $\epsilon = 0$. 
\begin{lemma}\label{lemma: F_extend}
\begin{itemize}
\item[(i)] Define 
\begin{equation}\label{eq: f0}
f(\tau,c,0):=\phi_f(c),
\end{equation}
the speed index function for the front.
\item[(ii)] Define
\begin{equation}\label{eq: g0}
g(\tau,c,0):= \f{1+\gamma}\left(\gamma + e^{-\f{(1+\gamma)\tau}{c}}\right)-\phi_f(c).
\end{equation}
Then $f$ and $g$ are continuous at $(\tau,c,0)$.
\end{itemize}
\end{lemma}

\begin{proof}
\begin{itemize}
\item[(i)] For $x\leq 0$, we have $e^{\f{\omega_1 x}{c}}\leq 1$ and $e^{\f{\omega_2 x}{c}}\leq 1$ so $C_x(x,c,\epsilon)$ is bounded by a constant independent of $\epsilon$. Moreover, $\left|\integ{y}{x-\f{\tau}{\epsilon}}[x]<K(y)>\right |\leq \integ{y}{-\infty}[x]<|K(y)|>$, which is integrable over $(-\infty,0]$. Hence, we apply the dominated convergence theorem. Using $\omega_1(\epsilon)\to 1$ and $\omega_2(\epsilon) \to 0$, it is clear that 
\begin{align}\label{eq: limit_Cx}
C_x(x,c,\epsilon) \longrightarrow \f{c}e^{\f{x}{c}},\qquad \integ{y}{x-\f{\tau}{\epsilon}}[x]<K(y)> \longrightarrow \integ{y}{-\infty}[x]<K(y)>.
\end{align}
Bringing the limit inside the integral, by definition, we have $f(\tau,c,\epsilon) \to \phi_f(c)$ as $\epsilon \to 0$. 
\item[(ii)] Write
\begin{align*}
g(\tau,c,\epsilon)&=\left(\int_{-\infty}^0 + \int_0^{\f{\tau}{\epsilon}}\right)C_x\left(x-\f{\tau}{\epsilon},c,\epsilon\right)\left(\integ{y}{x-\f{\tau}{\epsilon}}[x]<K(y)>\right)\,dx \\
&=\integ{x}{-\infty}[0]<C_x\left(x-\f{\tau}{\epsilon},c,\epsilon\right)\left(\integ{y}{x-\f{\tau}{\epsilon}}[x]<K(y)>\right)> \\
&\phantom{=}+\integ{x}{0}[\f{\tau}{\epsilon}]<C_x\left(x-\f{\tau}{\epsilon},c,\epsilon\right)\left[1-\left(\int_{-\infty}^{x-\f{\tau}{\epsilon}} + \int_x^\infty \right)K(y)\,dy\right]> \\
&=\f{\gamma}{1+\gamma}-\overline{C}(\tau,c,\epsilon)-\phi_f(c)+E_1(\tau,c,\epsilon)-E_2(\tau,c,\epsilon)-E_3(\tau,c,\epsilon), \numberthis \label{eq: g_simp}\\
\intertext{where $C(0,c,\epsilon)=\f{\gamma}{1+\gamma}$ is used and}
\overline{C}(\tau,c,\epsilon) &:=C\left(-\f{\tau}{\epsilon},c,\epsilon\right), \numberthis \label{eq: Cbar} \\
E_1(\tau,c,\epsilon)&:=\integ{x}{-\infty}[0]<C_x\left(x-\f{\tau}{\epsilon},c,\epsilon\right)\left(\integ{y}{x-\f{\tau}{\epsilon}}[x]<K(y)>\right)>,\numberthis \label{eq: E1} \\
E_2(\tau,c,\epsilon)&:=\integ{x}{0}[\f{\tau}{\epsilon}]<C_x\left(x-\f{\tau}{\epsilon},c,\epsilon\right)\left(\integ{y}{-\infty}[x-\f{\tau}{\epsilon}]<K(y)>\right)>-\phi_f(c),\numberthis \label{eq: E2} \\
E_3(\tau,c,\epsilon)&:=\integ{x}{0}[\f{\tau}{\epsilon}]<C_x\left(x-\f{\tau}{\epsilon},c,\epsilon\right)\left(\integ{y}{x}[\infty]<K(y)>\right)>.\numberthis \label{eq: E3}
\end{align*}
We recall that $\f{\omega_2(\epsilon)}{\epsilon} \to \omega_2'(0)=1+\gamma$ from \eqref{eq: limits_eig}-\eqref{eq: omega_calcs}. Moreover, the function $C$ is a difference of exponential functions with $e^{-\f{\omega_1(\epsilon) \tau}{c\epsilon}} \to 0$ since $\omega_1(\epsilon) \to 1$. Therefore,
$$\overline{C}(\tau,c,\epsilon)=\frac{1}{\omega_1-\omega_2}\left[\f{1-\omega_2}{\omega_1}e^{-\frac{\omega_1 \tau}{c\epsilon}}-\f{1-\omega_1}{\omega_2}e^{-\frac{\omega_2 \tau}{c\epsilon}}\right]\longrightarrow -\f{1+\gamma}e^{-\f{(1+\gamma)\tau}{c}}.
$$
The proof is complete by showing that error terms $E_1$, $E_2$, $E_3$ vanish as $\epsilon \to 0$.

For $E_1$, the function $C_x$ is bounded. Hence, the integrand is dominated by a constant multiple of $\integ{y}{-\infty}[x]<|K(y)|>$, which is integrable. By dominated convergence theorem, we may bring the limit inside the integral, where $C_x\left(x-\f{\tau}{\epsilon},c,\epsilon\right) \to 0$ as $\epsilon \to 0$.

For $E_2$, after a change of variable, we may write
$$
E_2(\tau,c,\epsilon)=\integ{x}{-\infty}[0]<C_x(x,c,\epsilon)\left(\integ{y}{-\infty}[x]<K(y)>\right)\chi_{(-\f{\tau}{\epsilon},0]}(x)>-\phi_f(c).
$$
Again, by dominated convergence theorem, we may bring the limit inside the integral, where $C_x(x,c,\epsilon)\chi_{(-\f{\tau}{\epsilon},0]}(x) \to \f{e^{\f{x}{c}}}{c}$, and the result follows by the definition of $\phi_f(c)$ in \eqref{eq: phi_def}.

For $E_3$, similar to $E_1$, we find that $C_x$ is bounded so the integrand is dominated by a constant multiple of $\integ{y}{x}[\infty]<|K(y)|>$, which is integrable over $[0,\infty)$. Then $C_x\left(x-\f{\tau}{\epsilon},c,\epsilon\right) \to 0$ as $\epsilon \to 0$. This completes the proof.
\end{itemize}
\end{proof}

\subsubsection{The Continuous Extension of $DF$}
Having established continuous extensions of $f$ and $g$ at $\epsilon=0$, we need to do the same for the partial derivatives. The calculations are tedious applications of dominated convergence theorem, which we leave for Appendix \ref{append_A}. The following lemma summarizes the analysis.
\begin{lemma}\label{lemma: Fpartial_cont}
For $\epsilon=0$, define
\begin{align*}
(i) \qquad f_\tau(\tau,c,0)&:=0, \\
f_c(\tau,c,0)&:= \phi_f'(c),\\
f_\epsilon(\tau,c,0)&:=2\phi_f(c)-\f{c}\integ{x}{-\infty}[0]<|x|\left(e^{\f{x}{c}}+1\right)K(x)>\\
(ii) \qquad g_\tau(\tau,c,0)&=-\f{c}e^{-\f{(1+\gamma)\tau}{c}},\\
g_c(\tau,c,0)&=\f{\tau}{c^2}e^{-\f{(1+\gamma)\tau}{c}}-\phi_f'(c),\\
g_\epsilon(\tau,c,0)&=\left(2-\f{\tau}{c}\right)e^{-\f{(1+\gamma)\tau}{c}}-\f{c}e^{-\f{(1+\gamma)\tau}{c}}\integ{x}{-\infty}[0]<|x|K(x)>\\
&\phantom{=}-\left\lbrace 2\phi_f(c)-\f{c}\integ{x}{-\infty}[0]<|x|\left(e^{\f{x}{c}}+1\right)K(x)> \right\rbrace \\
&\phantom{=}+\left\lbrace \f{c}e^{-\f{(1+\gamma)\tau}{c}}\integ{x}{0}[\infty]<|x|K(x)> \right\rbrace.
\end{align*}
Then $Df$ and $Dg$ are continuous at $\epsilon=0$.
\end{lemma}
\begin{proof}
See Appendix \ref{append_A}.
\end{proof}

\subsection{Applying the Implicit Function Theorem}\label{subsec: IFT}
We establish the existence of a unique base point solution to the equation
$F(\tau,c,0)=\begin{pmatrix}
\theta \\ \theta
\end{pmatrix}$.
In the first component, from the definition of $f(\tau,c,0)$ in \eqref{eq: f0}, we must solve $\phi_f(c)=\theta$. By \cref{lemma: unique_front}, there is only one positive solution $c_0=c_f$.
 
In the second component, using the definition of $g(\tau,c,0)$ from \eqref{eq: g0}, the equation $g(\tau,c_f,0)=\theta$ reduces to
$$
\f{1+\gamma}\left(\gamma + e^{-\f{(1+\gamma)\tau}{c_0}}\right)-\underbrace{\phi_f(c_0)}_{=\theta}=\theta,
$$
which is solved uniquely by $e^{-\f{(1+\gamma)\tau_0}{c_f}}=2\theta(1+\gamma)-\gamma$, or
\begin{equation}\label{eq: tau0}
\tau_0=-\frac{c_0}{1+\gamma}\ln(2\theta(1+\gamma)-\gamma).
\end{equation}
Note that $0<2\theta(1+\gamma)-\gamma<1$ by (H1). Hence, the base point $(\tau_0,c_0,\epsilon_0)=(\tau_0,c_f,0)$ is well-defined and unique. We are now able to apply a standard application of the implicit function theorem with the following lemma.
\begin{lemma}\label{lemma: IFT}
There exists $\delta_p>0$ and $\epsilon_p>0$ such that for each $\epsilon\leq \epsilon_p$, there exists a unique pair of $C^1$ curves $(\tau(\epsilon),c(\epsilon))$ such that $f(\tau(\epsilon),c(\epsilon),\epsilon)=g(\tau(\epsilon),c(\epsilon),\epsilon)=\theta$. The curves satisfy $|\tau(\epsilon)-\tau_0|<\delta_p$ and $|c(\epsilon)-c_0|<\delta_p$ for all $\epsilon\leq \epsilon_p$. The limits $\tau(\epsilon)\to \tau_0$ and $c(\epsilon)\to c_0$ hold. The derivative of the vector $(\tau(\epsilon),c(\epsilon))^T$ satisfies
\begin{equation}
 \begin{pmatrix}
\tau'(0) \\ c'(0)
\end{pmatrix}
=-J^{-1}(\tau_0,c_0,0)
\begin{pmatrix}
f_\epsilon(\tau_0,c_0,0) \\ g_\epsilon(\tau_0,c_0,0)
\end{pmatrix}.
\end{equation}
\end{lemma}
\begin{proof}
Clearly $F$ is $C^1$ at and around the base point, as demonstrated in this section. Moreover, by \cref{lemma: Fpartial_cont}, we have formulas for all partial derivatives of $f$ and $g$ at $(\tau_0,c_f,0)$. Since $f_\tau(\tau_0,c_f,0)=0$, we have 
\begin{align*}
\det J&=\begin{vmatrix}
0 &f_c(\tau_0,c_f,0)\\
g_\tau(\tau_0,c_f,0) &g_c(\tau_0,c_f,0)
\end{vmatrix}=-f_c(\tau_0,c_f,0)g_\tau(\tau_0,c_f,0) \numberthis \label{eq: det_J} \\
&=\f{c_f}\phi_f'(c_f)e^{-\f{(1+\gamma)\tau_0}{c_f}},
\end{align*}
which is negative since $\phi_f'(c_f)<0$ by \cref{lemma: unique_front}. Hence, the hypotheses for the implicit function theorem are met and the result follows. The formula for $(\tau'(0),c'(0))^T$ is routine. This completes the proof.
\end{proof}
Specifically, using $\phi_f(c_f)=\theta$ and again using $f_\tau(\tau_0,c_f,0)=0$,
\begin{align*}
\tau'(0)&=-\f{\det J}\Bigl\lbrace g_c(\tau_0,c_f,0)f_\epsilon(\tau_0,c_f,0)-f_c(\tau_0,c_f,0)g_\epsilon(\tau_0,c_f,0) \Bigr\rbrace,\numberthis \label{eq: taup}\\
c'(0)&=\f{\det J}g_\tau(\tau_0,c_f,0)f_\epsilon(\tau_0,c_f,0)=-\f{f_\epsilon(\tau_0,c_f,0)}{f_c(\tau_0,c_f,0)}\numberthis \label{eq: cp} \\
&=-\f{\phi_f'(c_f)}\left\lbrace 2\theta-\f{c_f}\integ{x}{-\infty}[0]<|x|\left(e^{\f{x}{c_f}}+1\right)K(x)>\right\rbrace
\end{align*}
by \cref{lemma: Fpartial_cont}. This concludes the proof of \cref{thm: first_order} and a description of $(a(\epsilon),c(\epsilon))$ in \cref{thm: E}.
\section{Threshold Requirements and Closeness of $\mathcal{S}_\epsilon$ to $\mathcal{S}_0$}\label{sec: thres}
In this section, we complete the proof of \cref{thm: E}. Specifically, we show that for $\epsilon$ sufficiently small, $U_\epsilon$ satisfies $U_\epsilon(z)<\theta$ on $(-\infty,0)\cup (a(\epsilon),\infty)$ and $U_\epsilon(z)>\theta$ on $(0,a(\epsilon))$. Moreover, we prove the stronger claim that $d_H(\mathcal{S}_\epsilon,\mathcal{S}_0)\to 0$, where $\mathcal{S}_\epsilon=\Bigl\{(U_\epsilon(z),Q_\epsilon(z)):z\in \R \Bigr\}$. We recall
$$
d_H(\mathcal{S}_\epsilon,\mathcal{S}_0)=\max\left\lbrace \sup_{(U,Q)\in \mathcal{S}_0}d\Bigl(\mathcal{S}_\epsilon,(U,Q)\Bigr),\sup_{(U_\epsilon,Q_\epsilon)\in \mathcal{S}_\epsilon}d\Bigl((U_\epsilon,Q_\epsilon),\mathcal{S}_0\Bigr) \right\rbrace
$$
and $d$ is the standard Euclidean metric. In other words, $\mathcal{S}_\epsilon$ and $\mathcal{S}_0$ are mutually as close to one another as we desire as $\epsilon\to 0$. Clearly, $\mathcal{S}_0$ is compact, and since $(U_\epsilon(\pm \infty),Q_\epsilon(\pm \infty))=(0,0)\in \mathcal{S}_0$, the supremums can be replaced by maximums.

We note that $\mathcal{S}_\epsilon$ and $\mathcal{S}_0$ describe sets, whereas the sub and super threshold regions describe dynamics in $z$. Moreover, since the system is nonlocal, points in $(U,Q)$ space can repeat. If $d_H(\mathcal{S}_\epsilon,\mathcal{S}_0) \to 0$, then near $\mathcal{M}_0^R$ and $\mathcal{M}_0^L$, the function $U_\epsilon(z)$ clearly stays away from the threshold.  However, near $\mathcal{F}_0$ and $\mathcal{B}_0$, we need to ensure $U_\epsilon(z)$ only crosses the threshold one time each. We guarantee this by showing $U_\epsilon'(z)\approx U_f'(z)$ and $U_\epsilon'(z)\approx U_b'(z-a(\epsilon))$ on large segments of the front and back, respectively. Therefore, $U_\epsilon(z)$ crosses the threshold transversely only twice.
\subsection{Preliminaries and Proof Outline}
Firstly, we define $(U,Q)\in \R^2$ points in the following way:
\begin{align*}
s_\epsilon(z)&:=(U_\epsilon(z),Q_\epsilon
(z))\in \mathcal{S}_\epsilon, \\
x_f(z)&:=(U_f(z),0)\in \mathcal{F}_0, \\
x_b(z)&:=(U_b(z),1-2\theta)\in \mathcal{B}_0, \\
m_R(Q)&:=(1-Q,Q)\in \widehat{\mathcal{M}}_0^R, \\
m_L(Q)&:=(-Q,Q)\in \widehat{\mathcal{M}}_0^L.
\end{align*}
We will also denote $a_\epsilon=a(\epsilon)$, $\tau_\epsilon=\tau(\epsilon)$,  and $c_\epsilon=c(\epsilon)$.
\subsubsection*{Sketch of Proof}
Let $\delta>0$. We show there exists $\epsilon_0=\epsilon(\delta) \ll 1$ such that $d_H(\mathcal{S}_\epsilon,\mathcal{S}_0)<3\delta$ for all $\epsilon<\epsilon_0$. Thus, the convergence $d_H(\mathcal{S}_\epsilon,\mathcal{S}_0)\to 0$ holds while the solutions satisfy the threshold requirements.

Choose a large, fixed positive constant $z_0=z_0(\delta)$ which satisfies
\begin{align*}
&\left(\int_{-\infty}^{-z_0}+\int_{z_0}^\infty \right)|K(x)|\,\mathrm{d}x<\delta,\numberthis \label{eq: z0_first} \\
&d\Bigl(x_f(z),m_L(0)\Bigr)<\f{\delta}{4} \text{ for }z<-z_0, \qquad d\Bigl( x_f(z),m_R(0) \Bigr)<\f{\delta}{4} \text{ for }z>z_0, \numberthis \label{eq: z0_second} \\
&d\Bigl(x_b(z-a_\epsilon),m_R(1-2\theta) \Bigr)<\f{\delta}{4} \text{ for } z-a_\epsilon<-z_0, \numberthis \label{eq: z0_third} \\
&d\Bigl(x_b(z-a_\epsilon),m_L(1-2\theta) \Bigr)<\f{\delta}{4} \text{ for } z-a_\epsilon>z_0.
\end{align*}
Since $U_b(z)=2\theta-U_f(z)$, estimates in \eqref{eq: z0_second} and \eqref{eq: z0_third} are equivalent. Essentially, we demand the asymptotic tails of the front and back are near the corner points of $\mathcal{S}_0$.

For $\epsilon$ sufficiently small (so $a_\epsilon-z_0>z_0$), we can create a disjoint union of the real numbers in the form
$$
\R=R_1(z_0)\cup R_2(\epsilon,z_0) \cup R_3(\epsilon,z_0) \cup R_4(\epsilon,z_0),
$$
where
\begin{equation*}
R_1(z_0)=(-\infty,z_0],\qquad R_2(\epsilon,z_0)=(z_0,a_\epsilon-z_0),\qquad R_3(\epsilon,z_0)=[a_\epsilon-z_0,a_\epsilon+z_0],\qquad R_4(\epsilon,z_0)=(a_\epsilon+z_0,+\infty).
\end{equation*}
As $\epsilon \to 0$, but $\epsilon \neq 0$, all intervals except $R_1$ shift as $a_\epsilon\to +\infty$, but the partitioning of the real numbers remains. For notation purposes, we will write
$$
\mathcal{S}_\epsilon=\mathcal{S}_\epsilon^1 \cup \mathcal{S}_\epsilon^2 \cup \mathcal{S}_\epsilon^3 \cup \mathcal{S}_\epsilon^4,
$$
where $\mathcal{S}_\epsilon^i$ represents $\mathcal{S}_\epsilon$ restricted to region $R_i$.

In order to understand our strategy, consider the following proposition, which shows that we can assess the closeness of $\mathcal{S}_\epsilon$ to $\mathcal{S}_0$ by assessing the closeness on all four pieces separately. This is crucial, especially near the corners of $\mathcal{S}_0$, where $\mathcal{S}_\epsilon$ takes on a mixture of fast and slow dynamics.
\begin{proposition}\label{prop: dH_regions}
Suppose $\max \Bigl\{d_H(\mathcal{S}_\epsilon^1,\mathcal{F}_0),\, d_H(\mathcal{S}_\epsilon^2,
\mathcal{M}_0^R),\, d_H(\mathcal{S}_\epsilon^3,\mathcal{B}_0),\, d_H(\mathcal{S}_\epsilon^4,\mathcal{M}_0^L)\Bigr\}<\eta$. Then $d_H(\mathcal{S}_\epsilon,\mathcal{S}_0)< \eta$.
\end{proposition}
\begin{proof}
First, let $(U,Q)\in \mathcal{S}_0$. If, say, $(U,Q)\in \mathcal{F}_0$, then 
$$
d\Bigl(\mathcal{S}_\epsilon,(U,Q)\Bigr)\leq d\Bigl(\mathcal{S}_\epsilon^1,(U,Q)\Bigr)\leq d_H(\mathcal{S}_\epsilon^1,\mathcal{F}_0)<\eta.
$$
Applying a similar argument to the cases $(U,Q)\in \mathcal{M}_0^R \,\cup \, \mathcal{B}_0\, \cup\,  \mathcal{M}_0^L$, it follows that $\displaystyle \sup_{(U,Q)\in \mathcal{S}_0} d\Bigl(\mathcal{S}_\epsilon,(U,Q)\Bigr)< \eta$.

On the other hand, let $s_\epsilon(z)\in \mathcal{S}_\epsilon$. If, say, $s_\epsilon(z)\in \mathcal{S}_\epsilon^1$, then
$$
d(s_\epsilon(z),\mathcal{S}_0)\leq d(s_\epsilon(z),\mathcal{F}_0)\leq d_H(\mathcal{S}_\epsilon^1,\mathcal{F}_0)<\eta.
$$
Applying a similar argument to the cases $s_\epsilon(z)\in \mathcal{S}_\epsilon^i$ for $i=2,\,3,\, 4$ we find \\
$\sup_{z\in \R}\, d(s_\epsilon(z),\mathcal{S}_0) < \eta$. Then by the definition, of $d_H$, we find $d_H(\mathcal{S}_\epsilon,\mathcal{S}_0)< \eta$.
\end{proof}

We proceed as follows (decreasing the $\epsilon$ tolerance at each step, if necessary). See Figures \ref{fig: Pulse_partition} and \ref{fig: Haus}.
\begin{itemize}
\item[$R_1$:] We show that $d(s_\epsilon(z),x_f(z))\to 0$ and $d(s_\epsilon'(z),x_f'(z))\to 0$ uniformly. Hence, there exists $\epsilon_1$ such that $U_\epsilon(z)<\theta$ on $(-\infty,0)$ and $U_\epsilon(z)>\theta$ on $(0,z_0]$ 
for all $\epsilon<\epsilon_1$. We demand $\epsilon_1$ be small enough so that so that $d_H(\mathcal{S}_\epsilon^1,\mathcal{F}_0)<\delta$ and $s_\epsilon(z_0) \in B_d(\mathcal{M}_0^R,\delta)$, the ball of radius $\delta$ around $\mathcal{M}_0^R$.
\item[$R_2$:] After the above occurs in $R_1$, we show there exists $\epsilon_2\leq \epsilon_1$ such that $s_\epsilon(z)\in B_d(\mathcal{M}_0^R,3\delta)$ for all $z\in R_2$ and $\epsilon<\epsilon_2$, as $Q_\epsilon(z)$ increases. This also implies that $U_\epsilon(z)>\theta$ for $z\in R_2$. We also show $d_H(\mathcal{S}_\epsilon^2,\mathcal{M}_0^R)<3\delta$ so all points $m_R(Q)$ have a nearby point $s_\epsilon(z)\in \mathcal{S}_\epsilon^2$.
\item[$R_3$:] We take a similar approach on $R_3$ as we did on $R_1$, except instead showing $s_\epsilon(z)$ (and its derivative) are uniformly close to $x_b(z-a_\epsilon)$, which is the back shifted to cross the threshold at $z=a_\epsilon$. Since $R_3$ slides with $\epsilon$, we cannot say the solution converges to the back uniformly on $R_3$. However, by introducing the characteristic function $\chi_{R_3}(z,\epsilon)$, we are able to prove that $d(s_\epsilon(z),x_b(z-a_\epsilon))\chi_{R_3}(z,\epsilon) \to 0$ and $d(s_\epsilon'(z),x_b'(z-a_\epsilon))\chi_{R_3}(z,\epsilon) \to 0$ uniformly over $\R$. Hence, there exists $\epsilon_3\leq \epsilon_2$ such that $U_\epsilon(z)>\theta$ on $[a_\epsilon-z_0,a_\epsilon)$ and $U_\epsilon(z)<\theta$ on $(a_\epsilon,a_\epsilon+z_0]$ for all $\epsilon<\epsilon_3$ and $z\in R_3(\epsilon,z_0)$. For $\epsilon_3$ small enough, we show $d_H(\mathcal{S}_\epsilon^3,\mathcal{B}_0)< \delta$ and $s_\epsilon(a_\epsilon+z_0) \in B_d(\mathcal{M}_0^L,\delta)$.
\item[$R_4$:] Similar to $R_2$, we show that there exists $\epsilon_4\leq \epsilon_3$ such that for all $\epsilon<\epsilon_4$, starting with $s_\epsilon(a_\epsilon+z_0)\in B_d(\mathcal{M}_0^L,\delta)$, the solution stays close to this ball during downward travel of mixed time scales. We divide the region further to emphasize that during strictly slow travel, the solution converges to the line $U=-Q$. As a result, $U_\epsilon(z)<\theta$ and $d_H(\mathcal{S}_\epsilon^4,\mathcal{M}_0^L)<3\delta$. Then we complete the proof of \cref{thm: E}.
\end{itemize}
\begin{figure}[H]
\centering
\includegraphics[width=100mm]{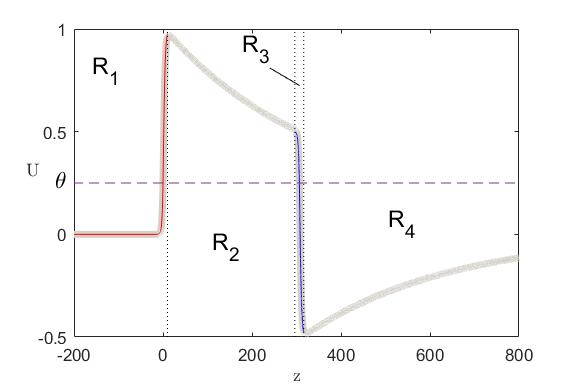}
\caption{A schematic of the pulse solution (grey), broken down by region. The pulse is close to the front (red) and back (blue) overlaying regions $R_1$ and $R_3$, respectively}
\label{fig: Pulse_partition}
\end{figure}
\begin{figure}[H]
\centering
\includegraphics[width=100mm]{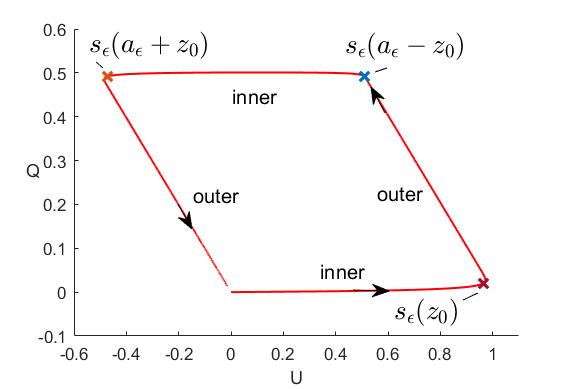}
\caption{The pulse as a homoclinic orbit. The pivot points $s_\epsilon(z_0)$, $s_\epsilon(a_\epsilon-z_0)$, and $s_\epsilon(a_\epsilon+z_0)$ are used to define the inner regions $R_1$ and $R_3$}
\label{fig: Haus}
\end{figure}

\subsection{Inner Region $R_1$}
Using the definition of $C_x$, we may write $U_\epsilon(z)$ in the form
\begin{equation}
U_\epsilon(z)=U_f(z)+A_1(z,\epsilon)-A_2(z,\epsilon), \label{eq: U_unif}
\end{equation}
where
\begin{align*}
A_1(z,\epsilon)&:=\integ{x}{-\infty}[z]<C_x(x-z,c_\epsilon,\epsilon)\left(\integ{y}{-\infty}[x]<K(y)>\right)>-U_f(z) \numberthis \label{eq: A_1} \\
&= \integ{x}{-\infty}[0]<\left[C_x(x,c_\epsilon,\epsilon)-\f{c_f}e^{\f{x}{c_f}}\right]\left(\integ{y}{-\infty}[x+z]<K(y)>\right)>,\\
A_2(z,\epsilon),\epsilon)&:=\integ{x}{-\infty}[z]<C_x(x-z,c_\epsilon,\epsilon)\left(\integ{y}{-\infty}[x-a_\epsilon]<K(y)>\right)>.\numberthis \label{eq: A_2}
\end{align*}
\begin{proposition}\label{prop: U_Uf_unif}
For $z_0>0$, we have $U_\epsilon(z) \to U_f(z)$ and $U_\epsilon'(z) \to U_f'(z)$ uniformly over $(-\infty,z_0]$ as $\epsilon\to 0$. Moreover, $Q_\epsilon\to 0$ and $Q_\epsilon' \to 0$ uniformly over the same interval.
\end{proposition}
\begin{proof}
We show $A_1$ and $A_2$ converge to zero uniformly over $(-\infty,z_0]$. For $A_1$, we have
\begin{align*}
|A_1(z,\epsilon)|&\leq \integ{x}{-\infty}[0]<\left|C_x(x,c_\epsilon,\epsilon)-\f{c_f}e^{\f{x}{c_f}}\right|\left(\integ{y}{-\infty}[x+z_0]<|K(y)|>\right)>.
\end{align*}
The integral is dominated by a constant multiple of $\integ{x}{-\infty}[0]<\integ{y}{-\infty}[x+z_0]<|K(y)|>> <+\infty$ so we may take the limit inside the integral, which converges to zero by \eqref{eq: limit_Cx}.

For $A_2$, we see that $C_x$ is bounded for $x\leq z$ so $A_2$ is bounded by a constant multiple of $\integ{x}{-\infty}[z_0]<\integ{y}{-\infty}[x-a_\epsilon]<|K(y)|>> \to 0$.

For $Q_\epsilon$, recalling its definition, as well as that of $D_x$, there are no difficulties establishing that $Q_\epsilon$ is uniformly bounded by a constant multiple of $\epsilon \left(\integ{x}{-\infty}[0]<\integ{y}{-\infty}[x+z_0]<|K(y)|>>\right) \to 0$. Obviously, $Q_\epsilon'$ is uniformly $O(\epsilon)$ over $\R$.
Finally, from 
\begin{align*}
&c_\epsilon U_\epsilon'(z)+U_\epsilon(z)+Q_\epsilon(z)=\integ{x}{z-a_\epsilon}[z]<K(x)>=\left(\int_{-\infty}^z-\int_{-\infty}^{z-a_\epsilon}\right)K(x)\,\mathrm{d}x,\\
&c_f U_f'(z)+U_f(z)=\integ{x}{-\infty}[z]<K(x)>,
\end{align*}
it follows that
\begin{equation}
\left|c_\epsilon U_\epsilon'(z)-c_f U_f'(z)\right|\leq |U_\epsilon(z)-U_f(z)|+|Q_\epsilon(z)|+ \integ{x}{-\infty}[z_0-a_\epsilon]<|K(x)|>,
\end{equation}
and with $c_\epsilon\to c_f$, clearly $U_\epsilon' \to U_f'$ uniformly.
\end{proof}
Hence, for a given fixed $z_0>0$, we may choose $\epsilon_1\leq \epsilon_p$ small enough so that $U_\epsilon(z)<\theta$ on $(-\infty,0)$ and $U_\epsilon(z)>\theta$ on $(0,z_0]$ for all $\epsilon<\epsilon_1$. If necessary, reduce $\epsilon_1$ so that $d(s_\epsilon(z),x_f(z))<\f{\delta}{2}$ for all $z\in R_1$. Then we conclude our analysis on $R_1$ with the following proposition.
\begin{proposition}\label{prop: R1_Hausdorff}
The estimate $d_H(\mathcal{S}_\epsilon^1,\mathcal{F}_0)< \delta$ holds for all $\epsilon<\epsilon_1$. Moreover, $s_\epsilon(z_0)\in B_d(\mathcal{M}_0^R,\delta)$.
\end{proposition}
\begin{proof}
Clearly,
\begin{equation}\label{eq: Haus_R1_first}
\sup_{z\in R_1}d(s_\epsilon(z),\mathcal{F}_0)\leq \sup_{z\in R_1}d\Bigl(s_\epsilon(z),x_f(z)\Bigr) <\f{\delta}{2}.
\end{equation}
On the other hand, let $x_f(z) \in \mathcal{F}_0$. If $z\leq z_0$, then
$$
d(\mathcal{S}_\epsilon^1,x_f(z))\leq d(s_\epsilon(z),x_f(z)) <\f{\delta}{2}.
$$
If $z>z_0$, then
$
d(\mathcal{S}_\epsilon^1,x_f(z) )\leq d(s_\epsilon(z_0),x_f(z) )$, which is bounded by
$$
d(x_f(z),m_R(0))+ d(m_R(0),x_f(z_0))+d(x_f(z_0),s_\epsilon(z_0) )<\f{\delta}{4}+\f{\delta}{4}+\f{\delta}{2}=\delta
$$
due to \eqref{eq: z0_second}. Consequently, we have $\displaystyle \sup_{z\in \R} d\Bigl(\mathcal{S}_\epsilon^1,x_f(z) \Bigr)< \delta$. Using \eqref{eq: Haus_R1_first}, by definition, we have $d_H(\mathcal{S}_\epsilon^1,\mathcal{F}_0)< \delta$.

For the second claim, note that
$$
d(s_\epsilon(z_0),\mathcal{M}_0^R)\leq d(s_\epsilon(z_0),m_R(0))\leq d(s_\epsilon(z_0),x_f(z_0))+d(x_f(z_0),m_R(0))<\delta.
$$
It follows that $s_\epsilon(z_0)\in B_d(\mathcal{M}_0^R,\delta)$.
\end{proof}
\subsection{Outer and Mixed Region $R_2$}
We first observe that the line $U-\gamma Q=0$ and $B_d(\mathcal{M}_0^R,\delta)$ do not intersect. Hence, as long as $s_\epsilon(z)\in B_d(\mathcal{M}_0^R,\delta)$, the $Q_\epsilon$ component is strictly increasing. On the boundary points of $R_2$, we show that $Q_\epsilon$ can be controlled in the sense of $Q_\epsilon(z_0)\approx 0$ and $Q_\epsilon(a_\epsilon-z_0)\approx 1-2\theta$. In region $R_1$, we showed $Q_\epsilon(z_0)\to 0$. We will prove $Q_\epsilon(a_\epsilon-z_0)\to 1-2\theta$ in the proceeding propositions.

Write $Q_\epsilon(z)$ in the following way:
\begin{align*}
Q_\epsilon(z)&=\left(\int_{-\infty}^0 + \int_0^{z}\right)D_x(x-z,c_\epsilon,\epsilon)\left(\integ{y}{x-a_\epsilon}[x]<K(y)>\right)\,dx \numberthis \label{eq: Q_R2} \\
&=\integ{x}{-\infty}[0]<D_x(x-z,c_\epsilon,\epsilon)\left(\integ{y}{x-a_\epsilon}[x]<K(y)>\right)>\\
&\phantom{=}+\integ{x}{0}[z]<D_x(x-z,c_\epsilon,\epsilon)\left[1-\left(\int_{-\infty}^{x-a_\epsilon}+\int_x^\infty \right)K(y)\,dy\right]> \\
&=\frac{1}{1+\gamma}- D(-z,c_\epsilon,\epsilon) + B_1(z,\epsilon)-B_2(z,\epsilon)-B_3(z,\epsilon),
\end{align*}
where $D(0,c_\epsilon,\epsilon)=\f{1+\gamma}$ from \eqref{eq: D0} and
\begin{align}
B_1(z,\epsilon)&:=\integ{x}{-\infty}[0]<D_x(x-z,c_\epsilon,\epsilon)\left(\integ{y}{x-a_\epsilon}[x]<K(y)>\right)>, \label{eq: B1}\\
B_2(z,\epsilon)&:=\integ{x}{0}[z]<D_x(x-z,c_\epsilon,\epsilon)\left(\integ{y}{-\infty}[x-a_\epsilon]<K(y)>\right)>, \label{eq: B2}\\
B_3(z,\epsilon)&:=\integ{x}{0}[z]<D_x(x-z,c_\epsilon,\epsilon)\left(\integ{y}{x}[\infty]<K(y)>\right)>.\label{eq: B3}
\end{align}
\begin{proposition}\label{prop: B_errors}
The functions $B_1\chi_{R_2}$, $B_2\chi_{R_2}$, $B_3\chi_{R_2}$ converge to zero uniformly over $\R$ as $\epsilon\to 0$.
\end{proposition}
\begin{proof}
For $B_1$, when $z\in R_2$, since $x\leq z$, the exponential functions in $D_x$ satisfy $e^{\f{\omega_1(x-z)}{c_\epsilon}}\leq 1$ and $e^{\f{\omega_2(x-z)}{c_\epsilon}}\leq 1$ so $|D_x(x-z,c_\epsilon,\epsilon)|\leq \f{2\epsilon}{c_\epsilon(\omega_1-\omega_2)}$ and 
\begin{equation}\label{eq: B1_zero}
|B_1(z,\epsilon)|\leq \f{2\epsilon}{c_\epsilon(\omega_1-\omega_2)}\integ{x}{-\infty}[0]<\integ{y}{x-a_\epsilon}[x]<|K(y)|>>
\end{equation}
so $B_1\chi_{R_2}$ is uniformly $O(\epsilon)$ over $\R$.

The proof for $B_3\chi_{R_2}$ is similar, except the double integral of $|K(y)|$ in \eqref{eq: B1_zero} is replaced with $\integ{x}{0}[a_\epsilon-z_0]<\integ{y}{x}[\infty]<|K(y)|>>$, which is uniformly bounded.

For $B_2$, we make a change of variable $x-z \mapsto x$ and since $z< a_\epsilon-z_0$ when $z\in R_2$,
\begin{align*}
|B_3(z,\epsilon)|&=\left| \integ{x}{-z}[0]<D_x(x,c_\epsilon,\epsilon)\left(\integ{y}{-\infty}[x+z-a_\epsilon]<K(y)>\right)> \right| \\
&\leq \integ{x}{-(a_\epsilon-z_0)}[0]<|D_x(x,c_\epsilon,\epsilon)|\left(\integ{y}{-\infty}[x-z_0]<|K(y)|>\right)> \\
&\leq \f{2\epsilon}{c_\epsilon(\omega_1-\omega_2)}\integ{x}{-(a_\epsilon-z_0)}[0]<\integ{y}{-\infty}[x-z_0]<|K(y)|>> \longrightarrow 0
\end{align*}
without dependence on $z$. Therefore, $B_2\chi_{R_2} \to 0$ uniformly on $\R$.
\end{proof}

We obtain the necessary estimate on $Q_\epsilon(a_\epsilon-z_0)$ in the following proposition and remark that our application of the implicit function theorem in \cref{sec: calc} is needed here as well.
\begin{proposition}\label{prop: Q_converge_R2}
The limit $Q_\epsilon(a_\epsilon-z_0)\to 1-2\theta$ holds.
\end{proposition}
\begin{proof}
Consider the form of $Q_\epsilon$ in \eqref{eq: Q_R2}. Since $B_1 \chi_{R_2}$, $B_2 \chi_{R_2}$, and $B_3 \chi_{R_2}$ vanish uniformly in the limit by \cref{prop: B_errors}, we focus on $\f{1+\gamma}-D(-z,c_\epsilon,\epsilon)$. The first term in $D(-z,c_\epsilon,\epsilon)$ satisfies $\left|-\f{\epsilon}{\omega_1(\omega_1-\omega_2)}e^{-\f{\omega_1 z}{c_\epsilon}} \right|\leq \f{\epsilon}{\omega_1(\omega_1-\omega_2)}\to 0$, leaving only the second term. Thus, after writing $a_\epsilon=\f{\tau_\epsilon}{\epsilon}$, we may write
\begin{align*}
\lim_{\epsilon\to 0}Q_\epsilon(a_\epsilon-z_0)&=\f{1+\gamma}-\lim_{\epsilon\to 0} \left[\f{\epsilon}{\omega_2(\omega_1-\omega_2)}e^{\f{\omega_2 z_0}{c_\epsilon}}e^{-\f{\omega_2\tau_\epsilon}{\epsilon c_\epsilon}}\right].
\end{align*}
Since $z_0$ has no dependence on $\epsilon$, we see that $\omega_2 z_0\to 0$ in the first exponential. Therefore, the above is equivalent to
$$
\f{1+\gamma}-\f{ \omega_2'(0)}e^{-\f{\omega_2'(0)\tau_0}{c_f}}=\f{1+\gamma}\left(1-e^{-\f{(1+\gamma)\tau_0}{c_f}}\right)=\f{1+\gamma}\left(1-(2\theta(1+\gamma)-\gamma)\right)=1-2\theta
$$
by the definition of $\tau_0$.
\end{proof}
By \cref{prop: Q_converge_R2}, we may ensure $\epsilon$ is sufficiently small so that 
\begin{equation}\label{eq: mr_estimates}
d(m_R(Q_\epsilon(z_0)),m_R(0))<\delta, \qquad  d(m_R(Q_\epsilon(a_\epsilon-z_0)),m_R(1-2\theta)) < \delta.
\end{equation}
\begin{proposition}\label{prop: R2_int_bound}
For all $z\in R_2$, the estimate $\left| \left(\int_{-\infty}^{z-a_\epsilon}+\int_z^\infty \right)K(x)\,\mathrm{d}x \right|<\delta$ holds.
\end{proposition}
\begin{proof}
Clearly 
$$
\left| \left(\int_{-\infty}^{z-a_\epsilon}+\int_z^\infty \right)K(x)\,\mathrm{d}x \right| \leq \left(\int_{-\infty}^{z-a_\epsilon}+\int_z^\infty \right)|K(x)|\, \mathrm{d}x\leq \left(\int_{-\infty}^{-z_0}+\int_{z_0}^\infty \right)|K(x)|\, \mathrm{d}x<\delta
$$
by \eqref{eq: z0_first}.
\end{proof}
The line $U=1-Q$ extended to all $Q\in \R$ has a corresponding $\delta$-ball described by a box with left and right boundaries as parallel lines $\ell_{\pm}^R(Q):=1-Q\pm\sqrt{2}\,\delta$. Clearly, $s_\epsilon(z_0)$ is between $\ell_{-}^R$ and $\ell_{+}^R$ since $s_\epsilon(z_0)\in B_d(\mathcal{M}_0^R,\delta)$. The following proposition provides us with information about the flow of $s_\epsilon(z)$.
\begin{proposition}\label{prop: R2_nearline}
For all $z\in R_2$, the solution satisfies $|U_\epsilon(z)-(1-Q_\epsilon(z))|<\sqrt{2}\,\delta$ or equivalently, $s_\epsilon(z)$ does not touch lines $\ell_{\pm}^R(Q_\epsilon(z))$ as $s_\epsilon(z)$ travels upward.
\end{proposition}
\begin{proof}
As $Q_\epsilon$ increases, the only way for $s_\epsilon$ to touch $\ell_{-}^R$ requires $1-Q_\epsilon(z_c)=U_\epsilon(z_c)+\sqrt{2}\,\delta$ and $U_\epsilon'(z_c)<0$, which we show leads to a contradiction. From $U_\epsilon(z)=-c_\epsilon U_\epsilon'(z)-Q_\epsilon(z)+\integ{x}{z-a_\epsilon}[z]<K(x)>$, we find
\begin{align*}
U_\epsilon(z_c)&>-Q_\epsilon(z_c)+\integ{x}{z_c-a_\epsilon}[z_c]<K(x)>=-Q_\epsilon(z_c)+1-\left(\int_{-\infty}^{z_c-a_\epsilon}+\int_{z_c}^\infty \right)K(x)\,\mathrm{d}x \\
&=U_\epsilon(z_c)+\sqrt{2}\,\delta -\left(\int_{-\infty}^{z_c-a_\epsilon}+\int_{z_c}^\infty \right)K(x)\,\mathrm{d}x >U_\epsilon(z_c)+\left(\sqrt{2}-1\right)\delta
\end{align*}
by \cref{prop: R2_int_bound}, which is impossible. The solution can flow into $\ell_{+}^R$ in two ways with $1-Q_\epsilon(z_c)=U_\epsilon(z_c)-\sqrt{2}\, \delta$. If $U_\epsilon'(z_c)\geq 0$, then a reversal of the previous case shows $U_\epsilon(z_c)\leq U_\epsilon(z_c)+\left(1-\sqrt{2}\right)\delta$, which is impossible. The final possibility is $\f{\mathrm{d}Q_\epsilon}{\mathrm{d}U_\epsilon}=\f{Q_\epsilon'(z_c)}{U_\epsilon'(z_c)}\leq -1$ (so  the tangent line and $\ell_+^R$ both have negative slope, but the tangent line is steeper). Since $Q_\epsilon'(z)$ is positive and uniformly $O(\epsilon)$, we rule this out by finding a negative upper bound on $U_\epsilon'(z_c)$. Simply plugging in,
\begin{align*}
c_\epsilon U_\epsilon'(z_c)&=-U_\epsilon(z_c)-Q_\epsilon(z_c)+1-\left(\int_{-\infty}^{z_c-a_\epsilon}+\int_{z_c}^\infty \right)K(x)\,\mathrm{d}x <-\left(\sqrt{2}-1\right)\delta.
\end{align*}
Therefore, $\f{Q_\epsilon'(z_c)}{U_\epsilon'(z_c)}>-\f{c_\epsilon Q_\epsilon'(z_c)}{\left(\sqrt{2}-1\right)\delta}$. With $\delta$ fixed, if necessary, we can reduce $\epsilon_2$ so that $Q_\epsilon'(z)$ is uniformly small enough to lead to a contradiction for all $\epsilon<\epsilon_2$. In conclusion, $s_\epsilon$ stays between $\ell_{-}^R$ and $\ell_{+}^R$ during upward travel.
\end{proof}
We conclude this section with the following proposition, which also implies $U_\epsilon(z)>\theta$ in $R_2$.
\begin{proposition}\label{prop: R2_Haus_proof}
For $\epsilon$ sufficiently small, the estimate $d_H(\mathcal{S}_\epsilon^2,\mathcal{M}_0^R)< 3\delta$ holds.
\end{proposition}
\begin{proof}
Let $m_R(Q)\in \mathcal{M}_0^R$. If $Q=Q_\epsilon(z)$ for some $z\in R_2$, then $d(\mathcal{S}_\epsilon^2,m_R(Q_\epsilon(z))) \leq d(s_\epsilon(z),m_R(Q_\epsilon(z)))< \sqrt{2}\,\delta$ by \cref{prop: R2_nearline}. Similarly, if $s_\epsilon(z)\in \mathcal{S}_\epsilon^2$ is given with $m_R(Q_\epsilon(z))\in \mathcal{M}_0^R$, then $d(s_\epsilon(z),\mathcal{M}_0^R)\leq d(s_\epsilon(z),m_R(Q_\epsilon(z)))<\sqrt{2} \,\delta$. If $Q\neq Q_\epsilon(z)$ for some $z\in R_2$, either $Q\in \left(Q_\epsilon(a_\epsilon-z_0),1-2\theta\right]$ or $Q\in \left[0,Q_\epsilon(z_0)\right)$. In the first case, 
\begin{align*}
d(\mathcal{S}_\epsilon^2,m_R(Q))&\leq d(s_\epsilon(a_\epsilon-z_0),m_R(Q))\\
&\leq d(s_\epsilon(a_\epsilon-z_0),m_R(Q_\epsilon(a_\epsilon-z_0)))+d(m_R(Q_\epsilon(a_\epsilon-z_0)),m_R(Q)\Bigr).
\end{align*}
By \eqref{eq: mr_estimates}, we may conclude $d(m_R(Q_\epsilon(a_\epsilon-z_0)),m_R(Q))\leq d(m_R(Q_\epsilon(a_\epsilon-z_0)),m_R(1-2\theta))<\delta$, so combined, $d(\mathcal{S}_\epsilon^2,m_R(Q))<(\sqrt{2}+1)\delta<3\delta$. In the second case, a similar application of \eqref{eq: mr_estimates} shows $d(\mathcal{S}_\epsilon^2,m_R(Q))<3\delta$. In conclusion, we have considered all situations where $m_R(Q)\in\mathcal{M}_0^R$ and find
$$
\sup_{Q \in [0,1-2\theta]} d(\mathcal{S}_\epsilon^2,m_R(Q))<3\delta.
$$
The proof is complete by considering the situations where a given $s_\epsilon(z)$ satisfies $Q_\epsilon(z)>1-2\theta$ or $Q_\epsilon(z)<0$, respectively. In the first case,
\begin{align*}
d(s_\epsilon(z),\mathcal{M}_0^R )&\leq d(s_\epsilon(z),m_R(1-2\theta))\\
&\leq d(s_\epsilon(z),m_R(Q_\epsilon(z)))+d(m_R(Q_\epsilon(z)),m_R(1-2\theta)).
\end{align*}
Obviously, $d(s_\epsilon(z),m_R(Q_\epsilon(z)) )<\sqrt{2}\,\delta$ and $d(m_R(Q_\epsilon(z)),m_R(1-2\theta) )\leq d(m_R(Q_\epsilon(a_\epsilon-z_0)),m_R(1-2\theta))<\delta$
by \eqref{eq: mr_estimates}, so above is bounded by $3\delta$. The $Q_\epsilon(z)<0$ case is similar. In conclusion, we find
$$
\sup_{z \in R_2} d(s_\epsilon(z),\mathcal{M}_0^R )<3\delta.
$$
By, definition, $d_H(\mathcal{S}_\epsilon^2,\mathcal{M}_0^R)< 3\delta$.
\end{proof}
\subsection{Inner Region $R_3$}
Using $\int_{x-a_\epsilon}^x K=\left(\int_{-\infty}^x -\int_{-\infty}^{x-a_\epsilon}\right)K$ and $\int_{-\infty}^x K=1-\int_x^\infty K$, we write $s_\epsilon(z)$ in a form that describes it being uniformly close to a shifted back on $R_3$. For brevity, we will only write $U_\epsilon(z)$ and remark that $Q_\epsilon(z)$ breaks down similarly, but with $C$ replaced by $D$.
\begin{align*}
U_\epsilon(z)&=\integ{x}{-\infty}[z]<C_x(x-z,c_\epsilon,\epsilon)\left(\integ{y}{x-a_\epsilon}[x]<K(y)>\right)>\numberthis \label{eq: U_R3} \\
&=-\integ{x}{-\infty}[z]<C_x(x-z,c_\epsilon,\epsilon)\left(\integ{y}{-\infty}[x-a_\epsilon]<K(y)>\right)>\\
&\phantom{=}+\integ{x}{-\infty}[0]<C_x(x-z,c_\epsilon,\epsilon)\left(\integ{y}{-\infty}[x]<K(y)>\right)>\\
&\phantom{=}+\integ{x}{0}[z]<C_x(x-z,c_\epsilon,\epsilon)\left[1-\integ{y}{x}[\infty]<K(y)>\right]>\\
&=\underbrace{2\theta -U_f(z-a_\epsilon)}_{U_b(z-a_\epsilon)}-\Gamma_1(z,\epsilon)+\Gamma_2(z,\epsilon)+\Gamma_3(z,\epsilon)-\Gamma_4(z,\epsilon),
\end{align*}
where the error functions are defined by
\begin{align}
\Gamma_1(z,\epsilon)&:=\integ{x}{-\infty}[z]<C_x(x-z,c_\epsilon,\epsilon)\left(\integ{y}{-\infty}[x-a_\epsilon]<K(y)>\right)>-U_f(z-a_\epsilon),\label{eq: Gamma1} \\
\Gamma_2(z,\epsilon)&:=\integ{x}{-\infty}[0]<C_x(x-z,c_\epsilon,\epsilon)\left(\integ{y}{-\infty}[x]<K(y)>\right)>,\label{eq: Gamma2}\\
\Gamma_3(z,\epsilon)&:=\left(C(0,c_\epsilon,\epsilon)-C(-z,c_\epsilon,\epsilon)\right)-2\theta,\label{eq: Gamma3}\\
\Gamma_4(z,\epsilon)&:=\integ{x}{0}[z]<C_x(x-z,c_\epsilon,\epsilon)\left(\integ{y}{x}[\infty]<K(y)>\right)>. \label{eq: Gamma4}
\end{align}
\begin{proposition}\label{prop: Gamma_errors}
The functions $\Gamma_1\chi_{R_3}$, $\Gamma_2\chi_{R_3}$, $\Gamma_3\chi_{R_3}$, $\Gamma_4\chi_{R_3}$ converge to zero uniformly over $\R$ as $\epsilon\to 0$.
\end{proposition}
\begin{proof}
For $\Gamma_1$, we notice with the change of variable $x-a_\epsilon \mapsto x$ in the $x$ integral, it follows that $\Gamma_1(z,\epsilon)=A_1(z-a_\epsilon,\epsilon)$ from \eqref{eq: A_1}, where we showed $A_1 \to 0$ uniformly on $(-\infty,z_0]$. There is a one-to-one correspondence between $z\in R_3(\epsilon,z_0)$ and $z-a_\epsilon \in [-z_0,z_0]$. Hence, $\Gamma_1\chi_{R_3}\to 0$ uniformly by the proof of \cref{prop: U_Uf_unif}.

For $\Gamma_2$, with $z\in R_3$, we have 
\begin{align*}
|\Gamma_2(z,\epsilon)|&\leq \f{\alpha}{c_\epsilon(\omega_1-\omega_2)}\integ{x}{-\infty}[0]<\left[(1-\omega_2)e^{\f{\omega_1(x-z)}{c_\epsilon}}+(1-\omega_1)e^{\f{\omega_2(x-z)}{c_\epsilon}}\right]\left(\integ{y}{-\infty}[x]<e^{\rho y}>\right)> \numberthis \label{eq: Gamma2_inequal}\\
&=\f{\alpha}{\rho(\omega_1-\omega_2)}\left[\f{1-\omega_2}{\omega_1+\rho c_\epsilon}e^{-\f{\omega_1 z}{c_\epsilon}}+\f{1-\omega_1}{\omega_2+\rho c_\epsilon}e^{-\f{\omega_2 z}{c_\epsilon}}\right] \\
&\leq \f{\alpha}{\rho(\omega_1-\omega_2)}\left[\f{1-\omega_2}{\omega_1+\rho c_\epsilon}e^{-\f{\omega_1 (a_\epsilon-z_0)}{c_\epsilon}}+\f{1-\omega_1}{\omega_2+\rho c_\epsilon}\right] \longrightarrow 0
\end{align*}
For $\Gamma_3$, with $z\in R_3$, we see that $e^{-\f{\omega_i (a_\epsilon+z_0)}{c_\epsilon}}\leq e^{-\f{\omega_i z}{c_\epsilon}}\leq e^{-\f{\omega_i (a_\epsilon-z_0)}{c_\epsilon}}$ for $i=1\, ,2$ since both eigenvalues are positive. It follows from the definition of $C$ and $C(0,c_\epsilon,\epsilon)=\f{\gamma}{1+\gamma}$ that
\begin{equation}
\f{\gamma}{1+\gamma}-2\theta -\zeta(\epsilon,z_0)\leq \Gamma_3(z,\epsilon) \leq \f{\gamma}{1+\gamma}-2\theta -\zeta(\epsilon,-z_0), \label{eq: Gamma3_squeeze}
\end{equation}
where
\begin{align}
\zeta(\epsilon,z_0)&=\f{\omega_1-\omega_2}\left[\f{1-\omega_2}{\omega_1}e^{-\f{\omega_1 (a_\epsilon-z_0)}{c_\epsilon}}-\f{1-\omega_1}{\omega_2}e^{-\f{\omega_2 (a_\epsilon+z_0)}{c_\epsilon}}\right], \label{eq: zeta}
\end{align}
Clearly, the first exponential term decays to zero in $\zeta(\epsilon,\pm z_0)$. For the second exponential term, $\pm z_0\omega_2 \to 0$ so $\zeta(\epsilon,\pm z_0)$ converges to
\begin{align*}
-\f{1+\gamma}e^{-\f{(1+\gamma)\tau_0}{c_f}}=\f{\gamma}{1+\gamma}-2\theta,
\end{align*}
by the definition of $\tau_0$. Then $\Gamma_3 \chi_{R_3}\to 0$ uniformly by \eqref{eq: Gamma3_squeeze}.

Finally, for $\Gamma_4$, with $z\in R_3$, we have 
\begin{align*}
|\Gamma_4(z,\epsilon)|&\leq \f{\alpha}{c_\epsilon(\omega_1-\omega_2)}\integ{x}{0}[z]<\left[(1-\omega_2)e^{\f{\omega_1(x-z)}{c_\epsilon}}+(1-\omega_1)e^{\f{\omega_2(x-z)}{c_\epsilon}}\right]\left(\integ{y}{x}[\infty]<e^{-\rho y}>\right)> \numberthis \label{eq: Gamma4_inequal}\\
&=\f{\alpha}{\rho(\omega_1-\omega_2)}\left[\f{1-\omega_2}{\omega_1-\rho c_\epsilon}\left(e^{-\rho z}-e^{-\f{\omega_1 z}{c_\epsilon}}\right)+\f{1-\omega_1}{\omega_2-\rho c_\epsilon}\left(e^{-\rho z}-e^{-\f{\omega_2 z}{c_\epsilon}}\right)\right] \\
&\leq \f{\alpha}{\rho(\omega_1-\omega_2)}\left[\f{1-\omega_2}{|\omega_1-\rho c_\epsilon|}\left(e^{-\rho (a_\epsilon-z_0)}+e^{-\f{\omega_1 (a_\epsilon-z_0)}{c_\epsilon}}\right)+\f{2(1-\omega_1)}{|\omega_2-\rho c_\epsilon|}\right] \longrightarrow 0
\end{align*}
\end{proof}
\begin{lemma}\label{lemma: unif_aminb_a}
There exists $\epsilon_3\leq \epsilon_2$ such that for all $\epsilon<\epsilon_3$, the uniform estimates on $R_3$ of $U_\epsilon(z)\approx U_b(z-a_\epsilon)$, $U_\epsilon'(z)\approx U_b'(z-a_\epsilon)$, and $Q_\epsilon(z)\approx 1-2\theta$ hold. The estimates are as close as we desire as $\epsilon\to 0$.
\end{lemma}
\begin{proof}
The uniform limit $(U_\epsilon(z)-U_b(z-a_\epsilon)\chi_{R_3}\to 0$ follows directly from \cref{prop: Gamma_errors}. For $Q_\epsilon$, after writing it in a similar way to that of $U_\epsilon$ in \eqref{eq: U_R3}. A tedious adjustment of \cref{prop: Gamma_errors} results in $(Q_\epsilon(z)-\underbrace{Q_b(z-a_\epsilon)}_{=1-2\theta})\chi_{R_3} \to 0$ uniformly. Finally, the back satisfies
\begin{equation} \label{eq: Uback_diff}
c_f U_b'(z-a_\epsilon)+U_b(z-a_\epsilon)+(1-2\theta)=\integ{x}{z-a_\epsilon}[\infty]<K(x)>=1-\integ{x}{-\infty}[z-a_\epsilon]<K(x)>,
\end{equation}
while $U_\epsilon$ satisfies
\begin{equation}\label{eq: U_R3_diff}
c_\epsilon U_\epsilon'(z)+U_\epsilon(z)+Q_\epsilon(z)=1-\left(\int_z^\infty+\int_{-\infty}^{z-a_\epsilon}\right)K(x)\,\mathrm{d}x.
\end{equation}
Subtracting \eqref{eq: Uback_diff} from \eqref{eq: U_R3_diff} leads to the inequality
\begin{align*}
|c_\epsilon U_\epsilon'(z)-c_f U_b'(z-a_\epsilon)| &\leq |U_\epsilon(z)-U_b(z-a_\epsilon)|+|Q_\epsilon(z)-(1-2\theta)| \\
&\phantom{\leq}+\integ{x}{a_\epsilon-z_0}[\infty]<|K(x)|>\longrightarrow 0
\end{align*}
so $(U_\epsilon'(z)-U_b'(z-a_\epsilon))\chi_{R_3}\to 0$ uniformly.
\end{proof}
In conclusion, for $\epsilon<\epsilon_3$, the curve $(U_\epsilon,Q_\epsilon)$ is uniformly close to the (fast time) back in the singular homoclinic orbit, and because $U_\epsilon'(z)\approx U_b'(z-a_\epsilon)$ and $Q_\epsilon'(z)\approx Q_b'(z-a_\epsilon)=0$ as well (where $U_b'(0)<0$), the solution successfully crosses the threshold from above one time at $z=a_\epsilon$. We conclude this subsection with a statement about the Hausdorff distance which is similar to \cref{prop: R1_Hausdorff}. If necessary, ensure $d(s_\epsilon(z),x_b(z-a_\epsilon))\chi_{R_3}<\f{\delta}{2}$ for all $z\in\R$.
\begin{proposition}\label{prop: R3_Hausdorff}
The estimate $d_H(\mathcal{S}_\epsilon^3,\mathcal{B}_0)< \delta$ holds for all $\epsilon<\epsilon_3$. Moreover, $s_\epsilon(a_\epsilon+z_0)\in B_d(\mathcal{M}_0^L,\delta)$.
\end{proposition}
\begin{proof}
Firstly,
\begin{equation}\label{eq: Haus_R3_first}
\sup_{z\in R_3}d(s_\epsilon(z),\mathcal{B}_0)\leq \sup_{z\in R_3}d(s_\epsilon(z),x_b(z-a_\epsilon)) <\f{\delta}{2}.
\end{equation}
On the other hand, let $x_b(z) \in \mathcal{B}_0$. If $|z-a_\epsilon |\leq z_0$, then
$$
d(\mathcal{S}_\epsilon^3,x_b(z-a_\epsilon))\leq d(s_\epsilon(z),x_b(z-a_\epsilon)) <\f{\delta}{2}.
$$
If $z-a_\epsilon>z_0$, then
$
d(\mathcal{S}_\epsilon^3,x_b(z-a_\epsilon))\leq d(s_\epsilon(a_\epsilon+z_0),x_b(z-a_\epsilon))$, which is bounded by
$$
d(x_b(z-a_\epsilon),m_L(1-2\theta))+ d(m_L(1-2\theta),x_b(z_0))+d(x_b(z_0),s_\epsilon(a_\epsilon +z_0))<\delta
$$
using \eqref{eq: z0_second}. The $z-a_\epsilon<-z_0$ case is similar. It is clear that $d_H(\mathcal{S}_\epsilon^3,\mathcal{B}_0)<\delta$.

For the second claim, we observe that
$$
d(s_\epsilon(a_\epsilon +z_0),m_L(1-2\theta))\leq d(s_\epsilon(a_\epsilon+z_0),x_b(z_0))+d(x_b(z_0),m_L(1-2\theta))<\delta.
$$
It follows that $s_\epsilon(a_\epsilon+z_0)\in B_d(\mathcal{M}_0^L,\delta)$.
\end{proof}
\subsection{Outer and Mixed Region $R_4$}
We consider the slow system when $\epsilon=0$. After the inner region takes over instantaneously as the back travels to $\mathcal{M}_0^L$ at $\tau=\tau_0$, the zeroth order dynamics on $\mathcal{M}_0^L$ can be understood as solutions to the equations $U_L=-Q_L$ and $c_f\dot{Q}_L=-(1+\gamma)Q_L$ for $\tau\geq \tau_0$, subject to $Q_L(\tau_0)=1-2\theta$. The solution is given by 
$$
Q_L(\tau)=(1-2\theta)e^{-\f{(1+\gamma)(\tau-\tau_0)}{c_f}}.
$$
This solution is valid until $(U_L(+\infty),Q_L(+\infty))=(0,0)$. Hence, our strategy on $R_4$ is similar to the one on $R_2$ in that we show that $s_\epsilon(z)$ does not touch the lines $\ell_{\pm}^L=-Q \pm \sqrt{2}\delta$ as the solution travels downward under mixed time, but also, we show that during slow time, $s_\epsilon(z)$ converges to $\mathcal{M}_0^L$ along the tail as $z\to +\infty$.

With this said, let $\tau_1>0$. Then we may split $R_4$ into $R_4=R_4^1\cup R_4^2$, where $R_4^1:=\left(a_\epsilon+z_0,a_\epsilon+\f{\tau_1}{\epsilon}+z_0 \right)$ and  $R_4^2:=\left[a_\epsilon+\f{\tau_1}{\epsilon}+z_0,\infty \right)$ partition $R_4$ into mixed and outer regions. 
\subsubsection*{Mixed Region $R_4^1$}
The following propositions are similar to \cref{prop: R2_nearline,prop: R2_int_bound}.
\begin{proposition}\label{prop: R4_int_bound}
For all $z\in R_4$, the estimate $\left| \integ{x}{z-a_\epsilon}[z]<K(x)> \right|<\delta$ holds.
\end{proposition}
\begin{proof}
For all $z\in R_4$, we have
$$
\left|\integ{x}{z-a_\epsilon}[z]<K(x)>\right|\leq \integ{x}{z-a_\epsilon}[z]<|K(x)|>\leq \integ{x}{z_0}[\infty]<|K(x)|><\delta
$$
by \eqref{eq: z0_first}.
\end{proof}
\begin{proposition}\label{prop: R4_nearline}
For all $z\in R_4^1$, the solution satisfies $|U_\epsilon(z)+Q_\epsilon(z))|<\sqrt{2}\,\delta$ or equivalently, $s_\epsilon(z)$ does not touch lines $\ell_{\pm}^L(Q_\epsilon(z))$ as $s_\epsilon(z)$ travels downward.
\end{proposition}
\begin{proof}
As $Q_\epsilon$ decreases, the only way for $s_\epsilon$ to touch $\ell_{+}^L$ requires $-Q_\epsilon(z_c)=U_\epsilon(z_c)-\sqrt{2}\,\delta$ and $U_\epsilon'(z_c)>0$, which similar to \cref{prop: R2_nearline}, leads to the contradiction
$$
U_\epsilon(z_c)<U_\epsilon(z_c)-\sqrt{2}\,\delta+\left|\integ{x}{z_c-a_\epsilon}[z_c]<K(x)>\right|<U_\epsilon(z_c)-\left(\sqrt{2}-1\right)\delta.
$$
The solution can flow into $\ell_{-}^L$ when $-Q_\epsilon(z_c)=U_\epsilon(z_c)+\sqrt{2}\,\delta$ in two ways. If $U_\epsilon'(z_c)\leq 0$, then a reversal of the previous case shows $U_\epsilon(z_c)\geq U_\epsilon(z_c)+\left(\sqrt{2}-1\right)\delta$, which is impossible. The second possibility is $\f{\mathrm{d}Q_\epsilon}{\mathrm{d}U_\epsilon}=\f{Q_\epsilon'(z_c)}{U_\epsilon'(z_c)}\leq -1$. 
Plugging in,
\begin{align*}
c_\epsilon U_\epsilon'(z_c)&=\sqrt{2}\,\delta + \integ{x}{z-a_\epsilon}[z]<K(x)> >\left(\sqrt{2}-1\right)\delta.
\end{align*}
Since $Q_\epsilon'(z_c)<0$, we find $\f{Q_\epsilon'(z_c)}{U_\epsilon'(z_c)}>-\f{c_\epsilon |Q_\epsilon'(z_c)|}{\left(\sqrt{2}-1\right)\delta}$. With $Q_\epsilon'$ uniformly $O(\epsilon)$, we can guarantee this negative value is small enough to make $\f{Q_\epsilon'(z_c)}{U_\epsilon'(z_c)}\leq -1$ lead to a contradiction. In conclusion, $s_\epsilon$ stays between $\ell_{-}^L$ and $\ell_{+}^L$ during downward travel.
\end{proof}
As a result of the above propositions, we find that $s_\epsilon(z)$ starts in $B_d(\mathcal{M}_0^L,\delta)$ and does not leave during travel in region $R_4^1$, as the flow transitions from fast to slow in this region.
\subsubsection*{Outer Region $R_4^2$}
In preparation of showing $|U_\epsilon(z)+Q_\epsilon(z)|\chi_{R_4^2}\to 0$ uniformly, set $G(x,c_\epsilon,\epsilon)=C(x,c_\epsilon,\epsilon)+D(x,c_\epsilon,\epsilon)$ and $W_\epsilon=U_\epsilon+Q_\epsilon$. We will also use $\tau_\epsilon=\epsilon a_\epsilon \to \tau_0$. Explicitly,
\begin{align*}
G(x,c_\epsilon,\epsilon)&=\f{\omega_1-\omega_2}\left[\left(\f{1-\omega_2-\epsilon}{\omega_1}\right) e^{\f{\omega_1x}{c_\epsilon}}-\left(\f{1-\omega_1-\epsilon}{\omega_2}\right) e^{\f{\omega_2 x}{c_\epsilon}}\right], \numberthis \label{eq: G}\\
G_x(x,c_\epsilon,\epsilon)&=\f{c_\epsilon(\omega_1-\omega_2)}\left[(1-\omega_2-\epsilon) e^{\f{\omega_1x}{c_\epsilon}}-(1-\omega_1-\epsilon) e^{\f{\omega_2 x}{c_\epsilon}}\right].
\end{align*}  

Write
\begin{align*}
W_\epsilon(z)&=\integ{x}{-\infty}[0]<G_x(x-z,c_\epsilon,\epsilon)\left(\integ{y}{x-a_\epsilon}[x]<K(y)>\right)> \numberthis \label{eq: UplusQ_R4}\\
&\phantom{=}+\integ{x}{0}[a_\epsilon]<G_x(x-z,c_\epsilon,\epsilon)\left[1-\left(\int_{-\infty}^{x-a_\epsilon}+\int_x^\infty \right)K(y)\,dy\right]>\\
&\phantom{=}+\integ{x}{a_\epsilon}[z]<G_x(x-z,c_\epsilon,\epsilon)\left(\integ{y}{x-a_\epsilon}[x]<K(y)>\right)>\\
&=L_1(z,\epsilon)-L_2(z,\epsilon)-L_3(z,\epsilon)+L_4(z,\epsilon)+L_5(z,\epsilon),
\end{align*}
where
\begin{align}
L_1(z,\epsilon)&:=\integ{x}{-\infty}[0]<G_x(x-z,c_\epsilon,\epsilon)\left(\integ{y}{x-a_\epsilon}[x]<K(y)>\right)> \label{eq: L1}\\
L_2(z,\epsilon)&:=\integ{x}{0}[a_\epsilon]<G_x(x-z,c_\epsilon,\epsilon)\left(\integ{y}{-\infty}[x-a_\epsilon]<K(y)>\right)>, \label{eq: L2} \\
L_3(z,\epsilon)&:=\integ{x}{0}[a_\epsilon]<G_x(x-z,c(\epsilon),\epsilon)\left(\integ{y}{x}[\infty]<K(y)>\right)>, \label{eq: L3} \\
L_4(z,\epsilon)&:=\integ{x}{a_\epsilon}[z]<G_x(x-z,c_\epsilon,\epsilon)\left(\integ{y}{x-a_\epsilon}[x]<K(y)>\right)>, \label{eq: L4} \\
L_5(z,\epsilon)&=G(a_\epsilon-z,c_\epsilon,\epsilon)-G(-z,c_\epsilon,\epsilon). \label{eq: L5}
\end{align}

\begin{lemma}\label{lemma: UplusQ_R4_errors}
The functions $L_i\chi_{R_4^2}$ converge to zero uniformly over $\R$. Hence, $W_\epsilon(z)\chi_{R_4^2} \to 0$ uniformly. Moreover, $W_\epsilon'(z)\chi_{R_4^2} \to 0$ uniformly so the solution follows the line $U=-Q$ downward as $s_\epsilon(+\infty)=(0,0)$.
\end{lemma}
\begin{proof}
For $L_1$, $L_2$, and $L_3$, using $z-a_\epsilon\geq \f{\tau_1}{\epsilon}+z_0$ and $e^{\f{\omega_1 x}{c_\epsilon}}\leq e^{\f{\omega_1 \tau_\epsilon}{\epsilon c_\epsilon}}$ when $x\leq a_\epsilon=\f{\tau_\epsilon}{\epsilon}$, we note that $|G_x(x-z,c_\epsilon,\epsilon)|\leq \beta(\epsilon)$, where 
$$
\beta(\epsilon):=\f{c_\epsilon(\omega_1-\omega_2)}\left[(1-\omega_2-\epsilon)e^{-\f{\omega_1 (\tau_1+\epsilon z_0)}{\epsilon c_\epsilon}}+\left|1-\omega_1-\epsilon \right|\right]\longrightarrow 0.
$$ 
Hence,
\begin{align*}
|L_1(z,\epsilon)|&\leq \beta(\epsilon)\integ{x}{-\infty}[0]<\integ{y}{x-a_\epsilon}[x]<|K(y)|>> \longrightarrow 0, \\
|L_2(z,\epsilon)|&\leq \beta(\epsilon)\integ{x}{0}[a_\epsilon]<\integ{y}{-\infty}[x-a_\epsilon]<|K(y)|>>\longrightarrow 0,\\
|L_3(z,\epsilon)|&\leq \beta(\epsilon)\integ{x}{0}[a_\epsilon]<\integ{y}{x}[\infty]<|K(y)|>>\longrightarrow 0.
\end{align*}
For $L_4 \chi_{R_4^2}$, we have
\begin{align*}
|L_4(z,\epsilon)|&\leq \f{\alpha}{c_\epsilon(\omega_1-\omega_2)}\integ{x}{a_\epsilon}[z]<\left[(1-\omega_2-\epsilon) e^{\f{\omega_1 (x-z)}{c_\epsilon}}+|1-\omega_1-\epsilon| e^{\f{\omega_2 (x-z)}{c_\epsilon}}\right]\left(\integ{y}{x-a_\epsilon}[x]<e^{-\rho y}>\right)>\\
&=\f{\alpha (e^{\rho a_\epsilon}-1)}{\rho c_\epsilon(\omega_1-\omega_2)}\integ{x}{a_\epsilon}[z]<(1-\omega_2-\epsilon) e^{-\f{\omega_1 z}{c_\epsilon}}e^{\f{(\omega_1-\rho c_\epsilon) x}{c_\epsilon}}+|1-\omega_1-\epsilon| e^{-\f{\omega_2 z}{c_\epsilon}}e^{\f{(\omega_2-\rho c_\epsilon) x}{c_\epsilon}}>\\
&=\f{\alpha (1-e^{-\rho a_\epsilon})}{\rho (\omega_1-\omega_2)}\left\{\f{1-\omega_2-\epsilon}{\omega_1-\rho c_\epsilon} \left(e^{-\rho (z-a_\epsilon)}-e^{-\f{\omega_1 (z-a_\epsilon)}{c_\epsilon}}\right)\right. \tag{i} \label{R4_L4_i}\\
&\phantom{\f{\alpha (1-e^{-\rho a_\epsilon})}{\rho (\omega_1-\omega_2)}\Bigl\lbrace}\left.+\f{|1-\omega_1-\epsilon|}{\omega_2-\rho c_\epsilon}\left(e^{-\rho (z-a_\epsilon)}-e^{-\f{\omega_2 (z-a_\epsilon)}{c_\epsilon}}\right)\right\} \tag{ii} \label{R4_L4_ii}.
\end{align*}
In \eqref{R4_L4_ii}, we may bound the term by $2 \left|\f{1-\omega_1-\epsilon}{\omega_2-\rho c_\epsilon}\right| \to 0$. In \eqref{R4_L4_i}, we see that $z-a_\epsilon\geq \f{\tau_1}{\epsilon}+z_0$ so the term is bounded by a constant multiple of $\left(e^{-\f{\rho (\tau_1+\epsilon z_0)}{\epsilon}}+e^{-\f{\omega_1 (\tau_1+\epsilon z_0)}{\epsilon c_\epsilon}} \right)\to 0$.

Finally, for $L_5$, since $e^{-\f{\omega_i z}{c_\epsilon}}<e^{-\f{\omega_i (z-a_\epsilon)}{c_\epsilon}}$, we have
\begin{align*}
|L_5(z,\epsilon)|&< \f{2}{\omega_1-\omega_2}\left[\left(\f{1-\omega_2-\epsilon}{\omega_1}\right) e^{-\f{\omega_1 (z-a_\epsilon)}{c_\epsilon}}+\left|\f{1-\omega_1-\epsilon}{\omega_2}\right| e^{-\f{\omega_2 (z-a_\epsilon)}{c_\epsilon}}\right]\\
&\leq \f{2}{\omega_1-\omega_2}\left[\left(\f{1-\omega_2-\epsilon}{\omega_1}\right) e^{-\f{\omega_1 (\tau_1+\epsilon z_0)}{\epsilon c_\epsilon}}+\left|\f{1-\omega_1-\epsilon}{\omega_2}\right|\right] \longrightarrow 0.
\end{align*}
In conclusion, $W_\epsilon(z)\chi_{R_4^2}\to 0$ uniformly. For the derivative, by adding the $U_\epsilon$ and $Q_\epsilon$ equations, we find
\begin{align*}
c_\epsilon |W_\epsilon'(z)|\leq |W_\epsilon(z)|+\epsilon |U_\epsilon(z)-\gamma Q_\epsilon(z)|+\integ{x}{\f{\tau_1}{\epsilon}}[\infty]<|K(x)|> \longrightarrow 0.
\end{align*}
This concludes the proof.
\end{proof}
After $\epsilon_4\leq \epsilon_3$ is chosen so that $d(s_\epsilon(z),\mathcal{M}_0^L)<\delta$ for all $z\in R_4^2$, we have enough information to complete the proof of \cref{thm: E}.
\begin{proposition}\label{prop: R4_Hausdorff}
The limit $d_H(\mathcal{S}_\epsilon,\mathcal{S}_0)\to 0$ holds.
\end{proposition}
\begin{proof}
Since $s_\epsilon(a_\epsilon+z_0)\in B_d(\mathcal{M}_0^L,\delta)$, we may combine \cref{prop: R4_nearline,lemma: UplusQ_R4_errors} to show $d(\mathcal{S}_\epsilon^4,\mathcal{M}_0^L)<3\delta$ by a similar proof seen in \cref{prop: R2_Haus_proof} on $R_2$. It follows that for $\epsilon<\epsilon_4$, we have
$$
\max \Bigl\{d(\mathcal{S}_\epsilon^1,\mathcal{F}_0),d(\mathcal{S}_\epsilon^2,\mathcal{M}_0^R),d(\mathcal{S}_\epsilon^3,\mathcal{B}_0),d(\mathcal{S}_\epsilon^4,\mathcal{M}_0^L) \Bigr\}<3\delta.
$$
Hence, by \cref{prop: dH_regions}, it follows that $d_H(\mathcal{S}_\epsilon,\mathcal{S}_0)<3\delta$.
\end{proof}
In conclusion, we proved in this section that $U_\epsilon$ crosses the threshold only at $z=0$ and $z=a_\epsilon$. Therefore, the formal homoclinic orbit connecting $(u,q)=(0,0)$ to itself is a true solution to the original system with Heaviside nonlinearity structure. Moreover, the homoclinic orbit is as close to the singular homoclinic orbit as we desire. This completes the proof of \cref{thm: E}.
\section*{Discussion}
In the present study, we successfully completed a rigorous proof of the existence and uniqueness (modulo translation) of fast traveling pulse solutions to the standard nonlocal, singularly perturbed neural field system with Heaviside activation functions and linear feedback. In \cref{sec: calc}, we changed the time scale in the two speed index functions due to the $O\left(\f{\epsilon}\right)$ second threshold crossing point. This allowed us to proceed with first-order approximations of the differences between pulse and front speed, not seen in previous literature. We also showed that, provided they exist, traveling fronts with single threshold crossing points have speed index functions with unique positive roots. Therefore, the corresponding Evans function cannot have positive zeros and $\lambda=0$ is always simple.

In \cref{sec: thres}, we carefully broke down the solution behavior near the four pieces of the singular homoclinic orbit and proved that the solution converges to it in the sense of Hausdorff distance between sets. In general, this is a difficult problem, especially for nonlocal problems, due to differences in time scale and a lack of standard phase space. However, because the solutions have closed form, the proofs were possible with direct analysis of the formal solutions during travel in inner, outer, and mixed regions. Our result is robust since the exponentially decaying kernel can take on any shape, assuming synaptic couplings are homogeneous and overall excitatory (i.e. $\int_\R K=1>0$).

There are obvious new directions worthy of follow-up. Consider the same problem, but with sigmoidal activation. Closed form is not available so we cannot track a formal solution near the singular homoclinic orbit. However, the solution and wave speed are qualitatively similar to those in the Heaviside case. Therefore, we may be able to track a nonlinear operator $\mathcal{N}[u,q,c]$ for inputs near Heaviside solutions at various thresholds and find a fixed point, which is typically the main strategy for pulses \cite{Faye2015} and bumps \cite{AmariKishimoto} without closed form.

We also did not study stability rigorously. As a result of Sandstede \cite{Sandstede-EvansFunctions} proving an equivalence between exponential and spectral stability, Zhang \cite{Zhang-OnStability} deriving the Evans function (for the front and pulse) to classify the spectrum, and the Evans function analysis from Coombes and Owen \cite{CoombesOwen_Evans} and Pinto et al. \cite{Pinto2005}, the task of rigorously locating the real $O(\epsilon)$ eigenvalue should be manageable. In turn, the pulses are stable.

Finally, there are other avenues to pursue. How do our findings connect to stochastic neural fields? Can we study more complicated kernel structures such as those representing heterogenous neural fields, where propagation failure can occur \cite{kilpatrick2008traveling}? How do (nonconstant) external currents influence our findings? Can we apply similar methods modeling synaptic depression?  We hope that future studies can thoroughly answer some of these questions.
\appendixpage
\appendix
\section{The Continuous Extensions of $Df$ and $Dg$} \label{append_A}
We first calculate the following second partial derivatives:
\begin{align*}
C_{xx}(x,c,\epsilon)&=\f{c^2(\omega_1-\omega_2)}\left[\omega_1(1-\omega_2)e^{\f{\omega_1 x}{c}}-\omega_2(1-\omega_1)e^{\f{\omega_2 x}{c}}\right],\numberthis \label{eq: Cxx} \\
C_{xc}(x,c,\epsilon)&=-\f{c}C_x(x,c,\epsilon)-\f{x}{c^3(\omega_1-\omega_2)}\left[\omega_1(1-\omega_2)e^{\f{\omega_1 x}{c}}-\omega_2(1-\omega_1)e^{\f{\omega_2 x}{c}}\right], \numberthis \label{eq: Cxc} \\
&=-\f{c^2(\omega_1-\omega_2)}\left[\left(1+\f{x\omega_1}{c}\right)(1-\omega_2)e^{\f{\omega_1 x}{c}}-\left(1+\f{x\omega_2}{c}\right)(1-\omega_1)e^{\f{\omega_2 x}{c}}\right], \\
C_{x\epsilon}(x,c,\epsilon)&= -\f{\omega_1'-\omega_2'}{\omega_1-\omega_2}C_x(x,c,\epsilon)\numberthis \label{eq: Cxepsilon}\\
&+\f{c(\omega_1-\omega_2)}\left[\left(\f{\omega_1' x}{c}(1-\omega_2)-\omega_2'\right)e^{\f{\omega_1 x}{c}}-\left(\f{\omega_2' x}{c}(1-\omega_1)-\omega_1'\right)e^{\f{\omega_2 x}{c}}\right]\\
&=\f{c(\omega_1-\omega_2)}\left\lbrace \left[\left(\f{\omega_1' x}{c}-\f{\omega_1'-\omega_2'}{\omega_1-\omega_2}\right)(1-\omega_2)-\omega_2'\right]e^{\f{\omega_1 x}{c}} \right.\\
&\phantom{=\f{c(\omega_1-\omega_2)}}-\left. \left[\left(\f{\omega_2' x}{c}-\f{\omega_1'-\omega_2'}{\omega_1-\omega_2}\right)(1-\omega_1)-\omega_1'\right]e^{\f{\omega_2 x}{c}}\right\rbrace .
\end{align*}
In preparation of applying the dominated convergence theorem, we prove the following limits.
\begin{proposition} \label{prop: second_partials_lims}
For fixed $x\in \R$, $\tau>0$, $c>0$, the following limits hold:
\begin{align}
&\lim_{\epsilon\to 0}\f{\epsilon}C_{xx}\left(x-\f{\tau}{\epsilon},c,\epsilon\right)=0, \label{eq: Cxx_overeps_lim} \\
&\lim_{\epsilon\to 0}\f{\tau}{\epsilon^2}C_{xx}\left(x-\f{\tau}{\epsilon},c,\epsilon\right) = -\f{(1+\gamma)\tau}{c^2}e^{-\f{(1+\gamma)\tau}{c}},\label{eq: Cxx_overepsq_lim}\\
&\lim_{\epsilon\to 0}C_{xc}(x,c,\epsilon)=-\f{c^2}\left(1+\f{x}{c}\right)e^{\f{x}{c}}, \label{eq: Cxc_lim} \\
&\lim_{\epsilon\to 0}C_{xc}\left(x-\f{\tau}{\epsilon},c,\epsilon\right)=0, \label{eq: Cxc_taueps_lim} \\
&\lim_{\epsilon\to 0}C_{x\epsilon}(x,c,\epsilon) =\f{c}\left[\left(1-\f{x}{c}\right)e^{\f{x}{c}}-1\right],\label{Cxeps_lim}\\
&\lim_{\epsilon\to 0}C_{x\epsilon}\left(x-\f{\tau}{\epsilon},c,\epsilon\right)=\f{c}\left(\f{(1+\gamma)\tau}{c}-1\right)e^{-\f{(1+\gamma)\tau}{c}}.\label{eq: Cxeps_taueps_lim}
\end{align}
\end{proposition}
\begin{proof}
For \eqref{eq: Cxx_overeps_lim}, since $\omega_2$ and $1-\omega_1$ are $O(\epsilon)$, it follows that
$$
\lim_{\epsilon\to 0}\f{\epsilon}C_{xx}\left(x-\f{\tau}{\epsilon},c,\epsilon\right)=\lim_{\epsilon\to 0}O\left(\f{\epsilon}\right)\left[O\left(e^{-\f{\omega_1 \tau}{c\epsilon}}\right)-O(\epsilon^2)\right]=0.
$$
For \eqref{eq: Cxx_overepsq_lim}, it is clear that $\f{\tau}{\epsilon^2}C_{xx}\left(x-\f{\tau}{\epsilon},c,\epsilon\right)=O\left(\f{\epsilon^2}\right)\left[O\left(e^{-\f{\omega_1 \tau}{c\epsilon}}\right)-O(\epsilon^2)\right]$. The first term vanishes, leaving
\begin{align*}
\lim_{\epsilon\to 0}\f{\tau}{\epsilon^2}C_{xx}\left(x-\f{\tau}{\epsilon},c,\epsilon\right) &=\lim_{\epsilon\to 0}\left[-\f{\tau \omega_2 (1-\omega_1)}{c^2\epsilon^2(\omega_1-\omega_2)}e^{\f{\omega_2 x}{c}}e^{-\f{\omega_2 \tau}{c\epsilon}}\right]=\f{\tau}{c^2}\omega_2'(0)\omega_1'(0)e^{-\f{\omega_2'(0)\tau}{c}}\\
&=-\f{(1+\gamma)\tau}{c^2}e^{-\f{(1+\gamma)\tau}{c}}.
\end{align*}
For \eqref{eq: Cxc_lim}, the second term of $C_{xc}(x,c,\epsilon)$ vanishes since $1-\omega_1\to 0$, leaving
\begin{align*}
\lim_{\epsilon\to 0} C_{xc}(x,c,\epsilon)&=\lim_{\epsilon\to 0}\left[-\f{c^2(\omega_1-\omega_2)}\left(1+\f{x\omega_1}{c}\right)(1-\omega_2)e^{\f{\omega_1 x}{c}}\right] \\
&=-\f{c^2}\left(1+\f{x}{c}\right)e^{\f{x}{c}}.
\end{align*}
For \eqref{eq: Cxc_taueps_lim}, the result follows from the proof of \eqref{eq: Cxc_lim}, with the observation that $\f{\tau \omega_1}{\epsilon}e^{-\f{\omega_1 \tau}{c\epsilon}}\to 0$.

For \eqref{Cxeps_lim}, the result follows from the definition of $C_{x\epsilon}(x,c,\epsilon)$ from \eqref{eq: Cxepsilon} and simply plugging in $\omega_1(0)$, $\omega_1'(0)$, $\omega_2(0)$, $\omega_2'(0)$, and simplifying.

Finally, for \eqref{eq: Cxeps_taueps_lim}, recalling the definition of $C_{x\epsilon}\left(x-\f{\tau}{\epsilon},c,\epsilon\right)$ from \eqref{eq: Cxepsilon}, the first term vanishes since $\f{\omega_1' \tau}{c\epsilon}e^{-\f{\omega_1 \tau}{c\epsilon}}\to 0$, leaving
\begin{align*}
\lim_{\epsilon\to 0}C_{x\epsilon}\left(x-\f{\tau}{\epsilon},c,\epsilon\right)&=-\lim_{\epsilon\to 0}\f{c(\omega_1-\omega_2)}\left[\left(\f{\omega_2' \left(x-\f{\tau}{\epsilon}\right)}{c}-\f{\omega_1'-\omega_2'}{\omega_1-\omega_2}\right)(1-\omega_1)-\omega_1'\right]e^{\f{\omega_2 \left(x-\f{\tau}{\epsilon}\right)}{c}}\\
&=\f{c}\left(\f{(1+\gamma)\tau}{c}-1\right)e^{-\f{(1+\gamma)\tau}{c}}.
\end{align*}
\end{proof}
\subsection{The Continuous Extension of $Df$}
Due to the partial derivatives of $C_x$ being bounded and $\integ{y}{-\infty}[x]<|K(y)|>$ being integrable over $(-\infty,0]$, we may differentiate and take limits inside the integral. In particular, standard calculations show
\begin{align*}
f_\tau(\tau,c,\epsilon)&=\f{\epsilon}\integ{x}{-\infty}[0]<C_x(x,c,\epsilon)K\left(x-\f{\tau}{\epsilon}\right)>,\numberthis \label{eq: ftau} \\
f_c(\tau,c,\epsilon)&=\integ{x}{-\infty}[0]<C_{xc}(x,c,\epsilon)\left(\integ{y}{x-\f{\tau}{\epsilon}}[x]<K(y)>\right)>,\numberthis \label{eq: fc}\\
f_\epsilon(\tau,c,\epsilon)&=\integ{x}{-\infty}[0]<C_{x\epsilon}(x,c,\epsilon)\left(\integ{y}{x-\f{\tau}{\epsilon}}[x]<K(y)>\right)> \numberthis \label{eq: feps} \\
&-\f{\tau}{\epsilon^2}\integ{x}{-\infty}[0]<C_x(x,c,\epsilon)K\left(x-\f{\tau}{\epsilon}\right)>.
\end{align*}
\begin{proof}[Proof of \cref{lemma: Fpartial_cont} (i)]
For $f_\tau$, we note that $\left|K\left(x-\f{\tau}{\epsilon}\right)\right|\leq \alpha e^{-\f{\rho\tau}{\epsilon}}e^{\rho x}$. It follows that $f_\tau(\tau,c,\epsilon)=O\left(\f{\epsilon}e^{-\f{\rho\tau}{\epsilon}}\right) \to 0$.

For $f_c$, we apply dominated convergence theorem and see that
\begin{align*}
\lim_{\epsilon\to 0} f_c(\tau,c,\epsilon)&=\integ{x}{-\infty}[0]<\lim_{\epsilon\to 0}\left[C_{xc}(x,c,\epsilon)\left(\integ{y}{x-\f{\tau}{\epsilon}}[x]<K(y)>\right)\right]> \\
&=\underbrace{-\f{c^2}\integ{x}{-\infty}[0]<\left(1+\f{x}{c}\right)e^{\f{x}{c}}\left(\integ{y}{-\infty}[x]<K(y)>\right)>}_{=\phi_f'(c)}
\end{align*}
by \eqref{eq: phi_def}, \eqref{eq: phip_def}, and \eqref{eq: Cxc_lim}. Finally, the second term of $f_\epsilon$ is $O\left(\f{\epsilon^2}e^{-\f{\rho\tau}{\epsilon}}\right)\to 0$. Therefore, by dominated convergence theorem again,
\begin{align*}
\lim_{\epsilon\to 0} f_\epsilon(\tau,c,\epsilon)&=\integ{x}{-\infty}[0]<\lim_{\epsilon\to 0} \left[C_{x\epsilon}(x,c,\epsilon)\left(\integ{y}{x-\f{\tau}{\epsilon}}[x]<K(y)>\right)\right]> \numberthis \label{eq: feps_lim} \\
&=\f{c}\integ{x}{-\infty}[0]<\left[\left(1-\f{x}{c}\right)e^{\f{x}{c}}-1\right]\left(\integ{y}{-\infty}[x]<K(y)>\right)> \\
&=\phi_f(c)+\left(c\phi_f'(c)+\phi_f(c)\right)-\f{c}\integ{x}{-\infty}[0]<\integ{y}{-\infty}[x]<K(y)>> \\
&=2\phi_f(c)-\f{c}\integ{x}{-\infty}[0]<|x|\left(e^{\f{x}{c}}+1\right)K(x)>
\end{align*}
after using the second form of $\phi_f'(c)$ from \eqref{eq: phip_def} and applying integration by parts to the double integral involving $K$.
\end{proof}
\subsection{The Continuous Extension of $Dg$}
Recalling $g=\f{\gamma}{1+\gamma}-\overline{C}-\phi_f+E_1-E_2-E_3$, we calculate the extension of each component. Then the proof of \cref{lemma: Fpartial_cont} (ii) is trivial.
\subsubsection*{The Continuous Extension of $D\overline{C}$} 
The partial derivatives of $C$ and $\overline{C}$ are as follows:
\begin{align*}
C_\epsilon(x,c,\epsilon)&=-\f{\omega_1'-\omega_2'}{\omega_1-\omega_2}C(x,c,\epsilon)+\f{e^{\f{\omega_1 x}{c}}}{\omega_1-\omega_2}\left[\left(\f{1-\omega_2}{\omega_1}\right)'+\f{\omega_1'x}{c}\left(\f{1-\omega_2}{\omega_1}\right)\right]\numberthis \label{C_eps}\\
&\phantom{=}-\f{e^{\f{\omega_2 x}{c}}}{\omega_1-\omega_2}\left[\left(\f{1-\omega_1}{\omega_2}\right)'+\f{\omega_2'x}{c}\left(\f{1-\omega_1}{\omega_2}\right)\right], \\
\overline{C}_\tau(\tau,c,\epsilon) &=-\f{\epsilon}C_x\left(-\f{\tau}{\epsilon},c,\epsilon\right)= -\f{c\epsilon(\omega_1-\omega_2)}\left[(1-\omega_2)e^{-\f{\omega_1\tau}{c\epsilon}}-(1-\omega_1)e^{-\f{\omega_2\tau}{c\epsilon}}\right],\numberthis \label{eq: C_tau} \\
\overline{C}_c(\tau,c,\epsilon)&=C_c\left(-\f{\tau}{\epsilon},c,\epsilon\right)=\f{\tau}{c^2\epsilon(\omega_1-\omega_2)}\left[(1-\omega_2)e^{-\f{\omega_1\tau}{c\epsilon}}-(1-\omega_1)e^{-\f{\omega_2\tau}{c\epsilon}}\right], \numberthis \label{eq: C_c} \\
\overline{C}_\epsilon(\tau,c,\epsilon)&=\f{\tau}{\epsilon^2}C_x\left(-\f{\tau}{\epsilon},c,\epsilon\right)+ C_\epsilon \left(-\f{\tau}{\epsilon},c,\epsilon\right) \numberthis \label{eq: Cbar_eps} \\
&=\f{\tau}{c\epsilon^2(\omega_1-\omega_2)}\left[(1-\omega_2)e^{-\f{\omega_1\tau}{c\epsilon}}-(1-\omega_1)e^{-\f{\omega_2\tau}{c\epsilon}}\right]-\f{\omega_1'-\omega_2'}{\omega_1-\omega_2}C\left(-\f{\tau}{\epsilon},c,\epsilon\right)\\
&\phantom{=}+\f{e^{-\f{\omega_1 \tau}{c\epsilon}}}{\omega_1-\omega_2}\left[\left(\f{1-\omega_2}{\omega_1}\right)'-\f{\omega_1'\tau}{c\epsilon}\left(\f{1-\omega_2}{\omega_1}\right)\right] \\
&\phantom{=}-\f{e^{-\f{\omega_2 \tau}{c\epsilon}}}{\omega_1-\omega_2}\left[\left(\f{1-\omega_1}{\omega_2}\right)'-\f{\omega_2'\tau}{c\epsilon}\left(\f{1-\omega_1}{\omega_2}\right)\right].
\end{align*}

\begin{proposition}\label{prop:  C_extend}
For $\epsilon=0$, define
\begin{align}
\overline{C}_\tau(\tau,c,0)&:=\f{c}e^{-\f{(1+\gamma)\tau}{c}}, \label{eq: C_tau_ext}\\
\overline{C}_c(\tau,c,0)&:=-\f{\tau}{c^2}e^{-\f{(1+\gamma)\tau}{c}}, \label{eq: C_c_ext}\\
\overline{C}_\epsilon(\tau,c,0)&:=\left(\f{\tau}{c}-2\right)e^{-\f{(1+\gamma)\tau}{c}}. \label{eq: C_eps_ext}
\end{align}
Then $\overline{C}_\tau$, $\overline{C}_c$, and $\overline{C}_\epsilon$ are continuous at $\epsilon=0$.
\end{proposition}
\begin{proof}
We first observe that $e^{-\f{\omega_1\tau}{c\epsilon}}$ decays to zero exponentially fast, while the terms multiplied by this function are algebraic. Hence, we may assume that all terms that are a constant multiple (dependent on $\epsilon$) of $e^{-\f{\omega_1\tau}{c\epsilon}}$ decay to zero.

For the second exponential function, we find that $e^{-\f{\omega_2\tau}{c\epsilon}}\to e^{-\f{\omega_2'(0)\tau}{c}}=e^{-\f{(1+\gamma)\tau}{c}}$. Therefore, the constant multiples of this function (involving $\omega_1$ and $\omega_2$ and their derivatives) need to be handled carefully in the limit.

For $\overline{C}_\tau$,
\begin{align*}
\lim_{\epsilon\to 0}\overline{C}_\tau(\tau,c,\epsilon)&=\lim_{\epsilon\to 0}\f{1-\omega_1}{c\epsilon(\omega_1-\omega_2)}e^{-\f{\omega_2\tau}{c\epsilon}}=\f{c}e^{-\f{(1+\gamma)\tau}{c}}. \numberthis \label{eq: Cbar_tau_lim}
\end{align*}
For $\overline{C}_c$,
\begin{align*}
\lim_{\epsilon\to 0}\overline{C}_c(\tau,c,\epsilon)&=-\lim_{\epsilon\to 0}\f{(1-\omega_1)\tau}{c^2\epsilon(\omega_1-\omega_2)}e^{-\f{\omega_2\tau}{c\epsilon}}=-\f{\tau}{c^2}e^{-\f{(1+\gamma)\tau}{c}}.\numberthis \label{eq: Cbar_c_lim}
\end{align*}
For $\overline{C}_\epsilon$, we group all constants that are multiplied by $e^{-\f{\omega_2\tau}{c\epsilon}}$ and find
\begin{align*}
\lim_{\epsilon\to 0}\overline{C}_\epsilon(\tau,c,\epsilon)&=\left[\lim_{\epsilon\to 0} J(\tau,c,\epsilon)\right]e^{-\f{(1+\gamma)\tau}{c}},
\end{align*}
where
\begin{align}
J(\tau,c,\epsilon)&= \f{\omega_1-\omega_2}\left\lbrace-\f{(1-\omega_1)\tau}{c\epsilon^2}\right. \label{eq: J_term1}\\
&\phantom{=}+\f{(\omega_1'-\omega_2')(1-\omega_1)}{\omega_2(\omega_1-\omega_2)}\phantom{\rbrace}\label{eq: J_term2} \\
&\phantom{=}\left. -\left[\left(\f{1-\omega_1}{\omega_2}\right)'-\f{\omega_2'\tau}{c\epsilon}\left(\f{1-\omega_1}{\omega_2}\right)\right]\right\rbrace \label{eq: J_term3}
\end{align}
Since $\f{\omega_1-\omega_2}\to 1$, we focus on the terms inside the braces. For \eqref{eq: J_term2}, there are no difficulties deducing from \eqref{eq: limits_eig} that
$$
\f{(\omega_1'-\omega_2')(1-\omega_1)}{\omega_2(\omega_1-\omega_2)}\longrightarrow -\f{2+\gamma}{1+\gamma}=-\left(1+\f{1+\gamma}\right).
$$
Also, using the form of $\f{1-\omega_1}{\omega_2}$ from \eqref{eq: OneMinusOmega1_Omega2}, it follows that
$
\left(\f{1-\omega_1}{\omega_2}\right)'=-\f{\omega_1'\gamma}{1+\gamma}\longrightarrow \f{\gamma}{1+\gamma}.
$

Finally, combining the term in \eqref{eq: J_term1} with the second term in \eqref{eq: J_term3}, we have
\begin{equation}
\left(\f{(1-\omega_1)\tau}{c\epsilon}\right)\left(\f{\omega_2'}{\omega_2}-\f{\epsilon}\right).
\end{equation}
The first term converges to $-\f{\omega_1'(0)\tau}{c}=\f{\tau}{c}$. Writing $\f{\omega_2'}{\omega_2}-\f{\epsilon}=\f{\epsilon\omega_2'-\omega_2}{\epsilon\omega_2}$, after two applications of L'Hospital's rule,
\begin{align*}
\lim_{\epsilon\to 0}\f{\epsilon\omega_2'-\omega_2}{\epsilon\omega_2}=\lim_{\epsilon\to 0}\f{\epsilon\omega_2''}{\epsilon\omega_2'+\omega_2}=\lim_{\epsilon\to 0}\f{\omega_2''+\epsilon\omega_2'''}{\epsilon\omega_2'' + 2\omega_2'}=1.
\end{align*}
Therefore, $J(\tau,c,\epsilon)\to \left\lbrace -\left(1+\f{1+\gamma}\right)-\f{\gamma}{1+\gamma}+\f{\tau}{c} \right\rbrace=\f{\tau}{c}-2$ so
\begin{align*}
\lim_{\epsilon\to 0}\overline{C}_\epsilon(\tau,c,\epsilon)&=\left[\lim_{\epsilon\to 0} J(\tau,c,\epsilon)\right]e^{-\f{(1+\gamma)\tau}{c}}=\left(\f{\tau}{c}-2\right)e^{-\f{(1+\gamma)\tau}{c}}.
\end{align*}
\end{proof}
\subsubsection*{The Continuous Extension of $DE_1$}
\begin{align*}
\f{\partial E_1}{\partial \tau}(\tau,c,\epsilon)&=\integ{x}{-\infty}[0]<\f{\partial}{\partial \tau}\left[C_x\left(x-\f{\tau}{\epsilon},c,\epsilon\right)\left(\integ{y}{x-\f{\tau}{\epsilon}}[x]<K(y)>\right)\right]> \numberthis \label{eq: E1_tau} \\
&=\integ{x}{-\infty}[0]<-\f{\epsilon} C_{xx}\left(x-\f{\tau}{\epsilon},c,\epsilon\right)\left(\integ{y}{x-\f{\tau}{\epsilon}}[x]<K(y)>\right)>\\
&\phantom{=}+\integ{x}{-\infty}[0]<C_x\left(x-\f{\tau}{\epsilon},c,\epsilon\right)\f{\epsilon}K\left(x-\f{\tau}{\epsilon}\right)>, \\
\f{\partial E_1}{\partial c}(\tau,c,\epsilon)&=\integ{x}{-\infty}[0]<C_{xc}\left(x-\f{\tau}{\epsilon},c,\epsilon\right)\left(\integ{y}{x-\f{\tau}{\epsilon}}[x]<K(y)>\right)>, \numberthis \label{eq: E1_c}\\
\f{\partial E_1}{\partial \epsilon}(\tau,c,\epsilon)&=\integ{x}{-\infty}[0]<\f{\partial}{\partial \epsilon}\left[C_x\left(x-\f{\tau}{\epsilon},c,\epsilon\right)\left(\integ{y}{x-\f{\tau}{\epsilon}}[x]<K(y)>\right)\right]> \numberthis \label{eq: E1_eps} \\
&=\integ{x}{-\infty}[0]<\f{\tau}{\epsilon^2}C_{xx}\left(x-\f{\tau}{\epsilon},c,\epsilon\right)\left(\integ{y}{x-\f{\tau}{\epsilon}}[x]<K(y)>\right)>\\
&\phantom{=}+\integ{x}{-\infty}[0]<C_{x\epsilon}\left(x-\f{\tau}{\epsilon},c,\epsilon\right)\left(\integ{y}{x-\f{\tau}{\epsilon}}[x]<K(y)>\right)>\\
&\phantom{=}-\integ{x}{-\infty}[0]<C_x\left(x-\f{\tau}{\epsilon},c,\epsilon\right)\f{\tau}{\epsilon^2}K\left(x-\f{\tau}{\epsilon}\right)>.
\end{align*}
The integration bounds over $E_1$ being intentionally cut off to $x\leq 0$ allow us to bring limits inside the integrals. Therefore, we can easily establish that $DE_1$ is continuous.
\begin{proposition}\label{prop: E1_C1}
For $\epsilon=0$, define
\begin{align}
\f{\partial E_1}{\partial \tau}(\tau,c,0)&:=0, \\
\f{\partial E_1}{\partial c}(\tau,c,0)&:=0, \\
\f{\partial E_1}{\partial \epsilon}(\tau,c,0)&:=-\f{c}e^{-\f{(1+\gamma)\tau}{c}}\integ{x}{-\infty}[0]<|x|K(x)>. 
\end{align}
Then $\f{\partial E_1}{\partial \tau}$, $\f{\partial E_1}{\partial c}$, and $\f{\partial E_1}{\partial \epsilon}$ are continuous at $\epsilon=0$.
\end{proposition}
\begin{proof}
In \eqref{eq: E1_tau}, by \eqref{eq: Cxx_overeps_lim}, the integrand of the first integral converges to zero, and in the second integral, $\f{\epsilon}K\left(x-\f{\tau}{\epsilon}\right) \to 0$ by the exponential decay of $K$. Therefore,
$$
\lim_{\epsilon\to 0}\f{\partial E_1}{\partial \tau}(\tau,c,\epsilon) =0.
$$
In \eqref{eq: E1_c}, by \eqref{eq: Cxc_taueps_lim},
\begin{align*}
\lim_{\epsilon\to 0} \f{\partial E_1}{\partial c}(\tau,c,\epsilon) &=\integ{x}{-\infty}[0]<\lim_{\epsilon\to 0}C_{xc}\left(x-\f{\tau}{\epsilon},c,\epsilon\right)\left(\integ{y}{x-\f{\tau}{\epsilon}}[x]<K(y)>\right)>=0.
\end{align*}
In \eqref{eq: E1_eps}, we observe that $\f{\partial E_1}{\partial \epsilon}$ splits into three integrals. The last integral is $O\left(\f{\tau}{\epsilon^2}e^{-\f{\rho\tau}{\epsilon}}\right) \to 0$ due to the exponential decay of $K$. For the first two integrals, we use the limits obtained in \eqref{eq: Cxx_overepsq_lim} and \eqref{eq: Cxeps_taueps_lim}. Namely,
\begin{align*}
&\lim_{\epsilon\to 0}\integ{x}{-\infty}[0]<\f{\tau}{\epsilon^2}C_{xx}\left(x-\f{\tau}{\epsilon},c,\epsilon\right)\left(\integ{y}{x-\f{\tau}{\epsilon}}[x]<K(y)>\right)> \numberthis \label{eq: E1_eps_term1}\\
&=\integ{x}{-\infty}[0]<\lim_{\epsilon \to 0}\f{\tau}{\epsilon^2}C_{xx}\left(x-\f{\tau}{\epsilon},c,\epsilon\right)\left(\integ{y}{x-\f{\tau}{\epsilon}}[x]<K(y)>\right)>\\
&=-\f{(1+\gamma)\tau}{c^2}e^{-\f{(1+\gamma)\tau}{c}}\integ{x}{-\infty}[0]<\integ{y}{-\infty}[x]<K(y)>>
\end{align*}
and
\begin{align*}
&\lim_{\epsilon\to 0}\integ{x}{-\infty}[0]<C_{x\epsilon}\left(x-\f{\tau}{\epsilon},c,\epsilon\right)\left(\integ{y}{x-\f{\tau}{\epsilon}}[x]<K(y)>\right)> \numberthis \label{eq: E1_eps_term2}\\
&=\integ{x}{-\infty}[0]<\lim_{\epsilon \to 0}C_{x\epsilon}\left(x-\f{\tau}{\epsilon},c,\epsilon\right)\left(\integ{y}{x-\f{\tau}{\epsilon}}[x]<K(y)>\right)>\\
&=-\f{c}\left(1-\f{(1+\gamma)\tau}{c}\right)e^{-\f{(1+\gamma)\tau}{c}}\integ{x}{-\infty}[0]<\integ{y}{-\infty}[x]<K(y)>>.
\end{align*}
Adding the resulting terms from \eqref{eq: E1_eps_term1} and \eqref{eq: E1_eps_term2}, the limit
\begin{align*}
\lim_{\epsilon\to 0}\f{\partial E_1}{\partial \epsilon}(\tau,c,\epsilon)&=-\f{c}e^{-\f{(1+\gamma)\tau}{c}}\integ{x}{-\infty}[0]<\integ{y}{-\infty}[x]<K(y)>> \numberthis \\
&=-\f{c}e^{-\f{(1+\gamma)\tau}{c}}\integ{x}{-\infty}[0]<|x|K(x)>
\end{align*}
holds. Hence, the partial derivatives of $E_1$ may be continuously extended at $\epsilon=0$.
\end{proof}
\subsubsection*{The Continuous Extension of $DE_2$}
After a change of variable, the $x$ integral of $E_2$ has an integrand dominated by a constant multiple of $\integ{y}{-\infty}[x]<|K(y)|>$, which is integrable over $(-\infty,0]$. Hence, we may apply Leibniz rule and compute the following partial derivatives:
\begin{align*}
\f{\partial E_2}{\partial \tau}(\tau,c,\epsilon)&= \f{\partial}{\partial \tau}\left[\integ{x}{-\f{\tau}{\epsilon}}[0]<C_x(x,c,\epsilon)\left(\integ{y}{-\infty}[x]<K(y)>\right)>-\phi_f(c)\right] \numberthis \label{eq: E2_tau} \\
&=\f{\epsilon} C_x\left(-\f{\tau}{\epsilon},c,\epsilon\right)\integ{y}{-\infty}[-\f{\tau}{\epsilon}]<K(y)>, \\
\f{\partial E_2}{\partial c}(\tau,c,\epsilon)&=\integ{x}{-\f{\tau}{\epsilon}}[0]<C_{xc}(x,c,\epsilon)\left(\integ{y}{-\infty}[x]<K(y)>\right)>-\phi_f'(c), \numberthis \label{eq: E2_c} \\
\f{\partial E_2}{\partial \epsilon}(\tau,c,\epsilon)&=-\f{\tau}{\epsilon^2} C_x\left(-\f{\tau}{\epsilon},c,\epsilon\right)\integ{y}{-\infty}[-\f{\tau}{\epsilon}]<K(y)> \numberthis \label{eq: E2_eps} \\
&\phantom{=}+\integ{x}{-\f{\tau}{\epsilon}}[0]<C_{x\epsilon}(x,c,\epsilon)\left(\integ{y}{-\infty}[x]<K(y)>\right)>.
\end{align*}
\begin{proposition}\label{prop: E2_C1}
For $\epsilon=0$, define
\begin{align}
\f{\partial E_2}{\partial \tau}(\tau,c,0)&:=0, \\
\f{\partial E_2}{\partial c}(\tau,c,0)&:=0, \\
\f{\partial E_2}{\partial \epsilon}(\tau,c,0)&:=2\phi_f(c)-\f{c}\integ{x}{-\infty}[0]<|x|\left(e^{\f{x}{c}}+1\right)K(x)>. 
\end{align}
Then $\f{\partial E_2}{\partial \tau}$, $\f{\partial E_2}{\partial c}$, and $\f{\partial E_2}{\partial \epsilon}$ are continuous at $\epsilon=0$.
\end{proposition}
\begin{proof}
By definition,
\begin{align*}
\f{\epsilon}C_x\left(-\f{\tau}{\epsilon},c,\epsilon)\right)&=\f{c\epsilon(\omega_1-\omega_2)}\left[(1-\omega_2)e^{-\f{\omega_1\tau}{c\epsilon}}-(1-\omega_1)e^{-\f{\omega_2\tau}{c\epsilon}}\right]\numberthis \label{eq: Cx_overeps_order} \\
&=O\left(\f{\epsilon}\right)\left[O(1)O\left(e^{-\f{\tau}{c\epsilon}}\right)-O(\epsilon)O(1)\right],
\end{align*}
which is $O(1)$. Due to the exponential bound on $K$, we have 
\begin{equation}
\integ{y}{-\infty}[-\f{\tau}{\epsilon}]<K(y)>=O\left(e^{-\f{\rho\tau}{\epsilon}}\right). \label{eq: K_int_order}
\end{equation}
It easily follows that 
\begin{equation}
\lim_{\epsilon\to 0}\f{\partial E_2}{\partial \tau}(\tau,c,\epsilon)=\lim_{\epsilon\to 0} O\left(e^{-\f{\rho\tau}{\epsilon}}\right)=0.
\end{equation} 
For the second limit,
\begin{align*}
\lim_{\epsilon\to 0} \f{\partial E_2}{\partial c}(\tau,c,\epsilon)&=\lim_{\epsilon\to 0}\integ{x}{-\f{\tau}{\epsilon}}[0]<C_{xc}(x,c,\epsilon)\left(\integ{y}{-\infty}[x]<K(y)>\right)>-\phi_f'(c) \\
&=-\f{c^2}\integ{x}{-\infty}[0]<\left(1+\f{x}{c}\right)e^{\f{x}{c}}\left(\integ{y}{-\infty}[x]<K(y)>\right)>-\phi_f'(c)=0.
\end{align*}
For the final limit, we first observe that as an immediate consequence of \eqref{eq: Cx_overeps_order} and \eqref{eq: K_int_order}, we have 
$$
-\f{\tau}{\epsilon^2} C_x\left(-\f{\tau}{\epsilon},c,\epsilon\right)\integ{y}{-\infty}[-\f{\tau}{\epsilon}]<K(y)>=O\left(\f{\epsilon}e^{-\f{\rho\tau}{\epsilon}}\right) \longrightarrow 0.
$$
Hence, the first term in \eqref{eq: E2_eps} vanishes as $\epsilon\to 0$, while the second term in \eqref{eq: E2_eps} converges to
\begin{align*}
&\phantom{=}\integ{x}{-\infty}[0]<\lim_{\epsilon\to 0}C_{x\epsilon}(x,c,\epsilon)\left(\integ{y}{-\infty}[x]<K(y)>\right)> \numberthis \\
&=\f{c}\integ{x}{-\infty}[0]<\left[\left(1-\f{x}{c}\right)e^{\f{x}{c}}-1\right]\left(\integ{y}{-\infty}[x]<K(y)>\right)> \\
&=2\phi_f(c)-\f{c}\integ{x}{-\infty}[0]<|x|\left(e^{\f{x}{c}}+1\right)K(x)>.
\end{align*}
\end{proof}
\subsubsection*{The Continuous Extension of $DE_3$}
Using Leibniz rule again, we may calculate the partial derivatives of $E_3$.
\begin{align*}
\f{\partial E_3}{\partial \tau}(\tau,c,\epsilon)&=\f{\epsilon}C_x(0,c,\epsilon)\integ{y}{\f{\tau}{\epsilon}}[\infty]<K(y)> \numberthis \label{eq: E3_tau} \\
&\phantom{=}-\f{\epsilon}\integ{x}{0}[\f{\tau}{\epsilon}]<C_{xx}\left(x-\f{\tau}{\epsilon},c,\epsilon\right)\left(\integ{y}{x}[\infty]<K(y)>\right)>,\\
\f{\partial E_3}{\partial c}(\tau,c,\epsilon)&=\integ{x}{0}[\f{\tau}{\epsilon}]<C_{xc}\left(x-\f{\tau}{\epsilon},c,\epsilon\right)\left(\integ{y}{x}[\infty]<K(y)>\right)>, \numberthis \label{eq: E3_c} \\
\f{\partial E_3}{\partial \epsilon}(\tau,c,\epsilon)&=-\f{\tau}{\epsilon^2}C_x(0,c,\epsilon)\integ{y}{\f{\tau}{\epsilon}}[\infty]<K(y)> \numberthis \label{eq: E3_eps} \\
&\phantom{=}+\f{\tau}{\epsilon^2}\integ{x}{0}[\f{\tau}{\epsilon}]<C_{xx}\left(x-\f{\tau}{\epsilon},c,\epsilon\right)\left(\integ{y}{x}[\infty]<K(y)>\right)> \\
&\phantom{=}+\integ{x}{0}[\f{\tau}{\epsilon}]<C_{x\epsilon}\left(x-\f{\tau}{\epsilon},c,\epsilon\right)\left(\integ{y}{x}[\infty]<K(y)>\right)>.
\end{align*}
\begin{proposition}\label{prop: E3_C1}
For $\epsilon=0$, define
\begin{align}
\f{\partial E_3}{\partial \tau}(\tau,c,0)&:=0, \\
\f{\partial E_3}{\partial c}(\tau,c,0)&:=0, \\
\f{\partial E_3}{\partial \epsilon}(\tau,c,0)&:=-\f{c}e^{-\f{(1+\gamma)\tau}{c}}\integ{x}{0}[\infty]<|x|K(x)>. 
\end{align}
Then $\f{\partial E_3}{\partial \tau}$, $\f{\partial E_3}{\partial c}$, and $\f{\partial E_3}{\partial \epsilon}$ are continuous at $\epsilon=0$.
\end{proposition}
\begin{proof}
In \eqref{eq: E3_tau}, for the first term, we have $\f{\epsilon}C_x(0,c,\epsilon)=O\left(\f{\epsilon}\right)$ and
$
\integ{y}{\f{\tau}{\epsilon}}[\infty]<K(y)>=O\left(e^{-\f{\rho\tau}{\epsilon}}\right)
$
so their product converges to zero. For the second term, by \eqref{eq: Cxx_overeps_lim}, in the limit,
$$
-\integ{x}{0}[\infty]<\underbrace{\lim_{\epsilon\to 0}\f{\epsilon}C_{xx}\left(x-\f{\tau}{\epsilon},c,\epsilon\right)}_{=0}\left(\integ{y}{x}[\infty]<K(y)>\right)>=0
$$
so combined, the result follows.

In \eqref{eq: E3_c}, by \eqref{eq: Cxc_taueps_lim},
\begin{align*}
\lim_{\epsilon\to 0} \f{\partial E_3}{\partial c}(\tau,c,\epsilon) &=\integ{x}{0}[\infty]<\underbrace{\lim_{\epsilon\to 0}C_{xc}\left(x-\f{\tau}{\epsilon},c,\epsilon\right)}_{=0}\left(\integ{y}{x}[\infty]<K(y)>\right)>=0.
\end{align*}
In \eqref{eq: E3_eps}, we observe that $\f{\partial E_3}{\partial \epsilon}$ splits into three terms. The first term is $O\left(\f{\epsilon^2}e^{-\f{\rho\tau}{\epsilon}}\right) \to 0$ due to the exponential decay of $\integ{y}{\f{\tau}{\epsilon}}[\infty]<K(y)>$. For the last two terms, we use the limits obtained in \eqref{eq: Cxx_overepsq_lim} and \eqref{eq: Cxeps_taueps_lim}. Namely,
\begin{align*}
&\lim_{\epsilon\to 0}\integ{x}{0}[\f{\tau}{\epsilon}]<\f{\tau}{\epsilon^2}C_{xx}\left(x-\f{\tau}{\epsilon},c,\epsilon\right)\left(\integ{y}{x}[\infty]<K(y)>\right)> \numberthis \label{eq: E3_eps_term1}\\
&=\integ{x}{0}[\infty]<\lim_{\epsilon \to 0}\f{\tau}{\epsilon^2}C_{xx}\left(x-\f{\tau}{\epsilon},c,\epsilon\right)\left(\integ{y}{x}[\infty]<K(y)>\right)>\\
&=-\f{(1+\gamma)\tau}{c^2}e^{-\f{(1+\gamma)\tau}{c}}\integ{x}{0}[\infty]<\integ{y}{x}[\infty]<K(y)>>
\end{align*}
and
\begin{align*}
&\lim_{\epsilon\to 0}\integ{x}{0}[\f{\tau}{\epsilon}]<C_{x\epsilon}\left(x-\f{\tau}{\epsilon},c,\epsilon\right)\left(\integ{y}{x}[\infty]<K(y)>\right)> \numberthis \label{eq: E3_eps_term2}\\
&=\integ{x}{0}[\infty]<\lim_{\epsilon \to 0}C_{x\epsilon}\left(x-\f{\tau}{\epsilon},c,\epsilon\right)\left(\integ{y}{x}[\infty]<K(y)>\right)>\\
&=-\f{c}\left(1-\f{(1+\gamma)\tau}{c}\right)e^{-\f{(1+\gamma)\tau}{c}}\integ{x}{0}[\infty]<\integ{y}{x}[\infty]<K(y)>>.
\end{align*}
Adding the resulting terms from \eqref{eq: E3_eps_term1} and \eqref{eq: E3_eps_term2}, the limit
\begin{align*}
\lim_{\epsilon\to 0}\f{\partial E_3}{\partial \epsilon}(\tau,c,\epsilon)&=-\f{c}e^{-\f{(1+\gamma)\tau}{c}}\integ{x}{0}[\infty]<\integ{y}{x}[\infty]<K(y)>> \numberthis \\
&=-\f{c}e^{-\f{(1+\gamma)\tau}{c}}\integ{x}{0}[\infty]<|x|K(x)>
\end{align*}
holds. Hence, the partial derivatives of $E_3$ may be continuously extended at $\epsilon=0$.
\end{proof}
The proof of \cref{lemma: Fpartial_cont} (ii) follows by combining the components of $Dg$.
\subsubsection*{Competing Interests}
The author declares that they have no competing interests.

\bibliographystyle{siam}
\begingroup
\setlength{\bibsep}{0pt}
\bibliography{Bibliography_arxiv_V2}

@Article{aggarwal2022v,
  author    = {Aggarwal, Adeeti and Brennan, Connor and Luo, Jennifer and Chung, Helen and Contreras, Diego and Kelz, Max B and Proekt, Alex},
  journal   = {Nat. Commun.},
  title     = {Visual evoked feedforward--feedback traveling waves organize neural activity across the cortical hierarchy in mice},
  year      = {2022},
  number    = {1},
  pages     = {4754},
  volume    = {13},
  publisher = {Nature Publishing Group UK London},
}

@article{Amari1977,
    title = {{Dynamics of pattern formation in lateral-inhibition type neural fields}},
    year = {1977},
    journal = {Biol. Cybernet.},
    author = {Amari, Shun ichi},
    number = {2},
    pages = {77--87},
    volume = {27},
    isbn = {0340-1200 (Print){\textbackslash}r0340-1200 (Linking)},
    doi = {10.1007/BF00337259},
    issn = {03401200},
    pmid = {911931}
}

@article{AmariKishimoto,
    title = {{Existence and stability of local excitations in homogeneous neural fields}},
    year = {1979},
    journal = {J. Math. Biol.},
    author = {Kishimoto, K. and Amari, S.},
    number = {4},
    pages = {303--318},
    volume = {7},
    isbn = {0303681217},
    doi = {10.1007/BF00275151},
    issn = {14321416},
    pmid = {224127}
}

@Article{aquino2012hemodynamic,
  author    = {Aquino, Kevin M and Schira, Mark M and Robinson, Peter A and Drysdale, Peter M and Breakspear, Michael},
  journal   = {PLoS Comput. Biol.},
  title     = {Hemodynamic traveling waves in human visual cortex},
  year      = {2012},
  number    = {3},
  pages     = {e1002435},
  volume    = {8},
  publisher = {Public Library of Science San Francisco, USA},
}

@Article{Atay2004,
  author   = {Atay, Fatihcan M. and Hutt, Axel},
  journal  = {SIAM J. Appl. Math.},
  title    = {Stability and bifurcations in neural fields with finite propagation speed and general connectivity},
  year     = {2004},
  issn     = {0036-1399},
  number   = {2},
  pages    = {644--666},
  volume   = {65},
  doi      = {10.1137/S0036139903430884},
}

@Article{avoli2016models,
  author    = {Avoli, Massimo and Jefferys, John GR},
  journal   = {J. Neurosci. Methods.},
  title     = {Models of drug-induced epileptiform synchronization in vitro},
  year      = {2016},
  pages     = {26--32},
  volume    = {260},
  publisher = {Elsevier},
}

@article{bressloff2014waves_wholebook,
    title = {{Waves in neural media}},
    year = {2014},
    journal = {Lecture Notes on Mathematical Modelling in the Life Sciences},
    author = {Bressloff, Paul C},
    pages = {},
    publisher = {Springer}
}

@Article{BressloffWebber,
  author   = {Bressloff, P. C. and Webber, M.},
  journal  = {SIAM J. Appl. Dyn. Syst.},
  title    = {{Front propagation in stochastic neural fields}},
  year     = {2011},
  issn     = {1536-0040},
  number   = {2},
  pages    = {708--740},
  volume   = {11},
  arxivid  = {arXiv:1406.2675v1},
  doi      = {10.1137/110851031},
  url      = {http://eprints.maths.ox.ac.uk/1449/},
}

@Article{burlakov_etal2025,
  author  = {Burlakov, Evgenii and Oleynik, Anna and Ponosov, Arcady},
  journal = {Mathematics (2227-7390)},
  title   = {Travelling Waves in Neural Fields with Continuous and Discontinuous Neuronal Activation.},
  year    = {2025},
  number  = {5},
  volume  = {13},
}

@Article{Chen1997,
  author    = {Chen, Xinfu},
  journal   = {Adv. Differ. Equ.},
  title     = {Existence, uniqueness, and asymptotic stability of traveling waves in nonlocal evolution equations},
  year      = {1997},
  number    = {1},
  pages     = {125--160},
  volume    = {2},
  publisher = {Khayyam Publishing, Inc.},
}

@Article{cook2022neural,
  author    = {Cook, Blake J and Peterson, Andre DH and Woldman, Wessel and Terry, John R},
  journal   = {Math. Neurosci. Appl.},
  title     = {Neural Field Models: A mathematical overview and unifying framework},
  year      = {2022},
  volume    = {2},
  publisher = {Episciences. org},
}

@book{coombes2014neural,
  title={Neural fields: theory and applications},
  author={Coombes, Stephen and beim Graben, Peter and Potthast, Roland and Wright, James},
  year={2014},
  publisher={Springer}
}

@article{CoombesOwen_Evans,
    title = {{Evans Functions for Integral Neural Field Equations with Heaviside Firing Rate Function}},
    year = {2004},
    journal = {SIAM J. Appl. Dyn. Syst.},
    author = {Coombes, S and Owen, M R},
    number = {4},
    pages = {574--600},
    volume = {3},
    doi = {10.1137/040605953},
    issn = {1536-0040}
}

@Article{CoombesSchmidt,
  author  = {Coombes, Stephen and Schmidt, Helmut},
  title   = {Neural fields with sigmoidal firing rates: Approximate solutions},
  journal = {Discrete Contin. Dyn. Syst. Ser. S},
  year    = {2010},
  volume  = {28},
  number  = {4},
  pages   = {1369--1379},
}

@Article{cowan2014personal,
  author    = {Cowan, Jack},
  title     = {A personal account of the development of the field theory of large-scale brain activity from 1945 onward},
  year      = {2014},
  pages     = {47--96},
  booktitle = {Neural fields},
  publisher = {Springer},
}

@Article{Dyson2019_MBE,
  author  = {Dyson, Alan},
  title   = {Traveling wave solutions to a neural field model with oscillatory synaptic coupling types},
  journal = {Math. Biosci. Eng.},
  year    = {2019},
  volume  = {16},
  number  = {mbe-16-02-035},
  pages   = {727--758},
  issn    = {1547-1063},
  doi     = {http://dx.doi.org/10.3934/mbe.2019035},
}

@Article{Dyson2020,
  author    = {Dyson, Alan},
  title     = {Existence and Uniqueness of Traveling Fronts in Lateral Inhibition Neural Fields with Sigmoidal Firing Rates},
  journal   = {SIAM J. Appl. Dyn. Syst.},
  year      = {2020},
  volume    = {19},
  number    = {3},
  pages     = {2194--2231},
  publisher = {SIAM},
}

@book{ermentrout2010mathematical,
  title={Mathematical foundations of neuroscience},
  author={Ermentrout, G Bard and Terman, David H},
  volume={35},
  year={2010},
  publisher={Springer Science \& Business Media}
}

@Article{ExistenceandUniqueness-ErmMcLeod,
  author  = {Ermentrout, G. B. and McLeod, J. B.},
  journal = {Proc. Roy. Soc. Edinburgh Sect. A},
  title   = {{Existence and uniqueness of travelling waves for a neural network}},
  year    = {1993},
  issn    = {14737124},
  number  = {July},
  pages   = {461--478},
  volume  = {123},
  doi     = {10.1017/S030821050002583X},
  url     = {http://journals.cambridge.org/abstract_S030821050002583X},
}

@article{ExploitingtheHamiltonian,
    title = {{Exploiting the Hamiltonian structure of a neural field model}},
    year = {2010},
    journal = {Phys. D},
    author = {Elvin, A. J. and Laing, C. R. and McLachlan, R. I. and Roberts, M. G.},
    number = {9},
    pages = {537--546},
    volume = {239},
    isbn = {01672789},
    doi = {10.1016/j.physd.2009.08.004},
    issn = {01672789}
}

@Article{Faye2013,
  author   = {Faye, G},
  journal  = {SIAM J. Appl. Dyn. Syst.},
  title    = {Existence and stability of traveling pulses in a neural field equation with synaptic depression},
  year     = {2013},
  issn     = {15360040},
  number   = {4},
  pages    = {2032--2067},
  volume   = {12},
  doi      = {10.1137/130913092},
  url      = {http://epubs.siam.org/doi/abs/10.1137/130913092},
}

@Article{Faye2015,
  author   = {Faye, Grégory and Scheel, Arnd},
  title    = {{Existence of pulses in excitable media with nonlocal coupling}},
  journal  = {Adv. Math.},
  year     = {2015},
  volume   = {270},
  pages    = {400--456},
  issn     = {10902082},
  arxivid  = {1311.6508},
  doi      = {10.1016/j.aim.2014.11.005},
  url      = {http://dx.doi.org/10.1016/j.aim.2014.11.005},
}

@Article{fenichel1979geometric,
  author    = {Fenichel, Neil},
  journal   = {J. Differ. Equ.},
  title     = {Geometric singular perturbation theory for ordinary differential equations},
  year      = {1979},
  number    = {1},
  pages     = {53--98},
  volume    = {31},
  publisher = {Academic Press},
}

@Article{folias2017,
  author    = {Folias, Stefanos E},
  journal   = {Phys. Rev. E},
  title     = {Traveling waves and breathers in an excitatory-inhibitory neural field},
  year      = {2017},
  number    = {3},
  pages     = {032210},
  volume    = {95},
  publisher = {APS},
}

@article{freestone2011data,
  title={A data-driven framework for neural field modeling},
  author={Freestone, Dean R and Aram, Parham and Dewar, Michael and Scerri, Kenneth and Grayden, David B and Kadirkamanathan, Visakan},
  journal={NeuroImage},
  volume={56},
  number={3},
  pages={1043--1058},
  year={2011},
  publisher={Elsevier}
}

@Article{grabot2025traveling,
  author    = {Grabot, Laetitia and Merholz, Garance and Winawer, Jonathan and Heeger, David J and Dugu{\'e}, Laura},
  journal   = {PLoS Comput. Biol.},
  title     = {Traveling waves in the human visual cortex: An MEG-EEG model-based approach},
  year      = {2025},
  number    = {4},
  pages     = {e1013007},
  volume    = {21},
  publisher = {Public Library of Science San Francisco, CA USA},
}

@article{guo2016existence,
  title={EXISTENCE AND NONEXISTENCE OF TRAVELING PULSES IN A LATERAL INHIBITION NEURAL NETWORK.},
  author={Guo, Yixin and Zhang, Aijun},
  journal={Discrete Contin. Dyn. Syst. Ser. S},
  volume={21},
  number={6},
  pages={1729--1755},
  year={2016}
}

@Article{HuttZhang-TravelingWave,
  author  = {Zhang, Linghai and Hutt, Axel},
  title   = {Traveling wave solutions of nonlinear scalar integral differential equations arising from synaptically coupled neuronal networks},
  journal = {J. Appl. Anal. Comput.},
  year    = {2014},
  volume  = {4},
  number  = {1},
  pages   = {1--68},
}

@article{jones1996tracking,
  title={Tracking invariant manifolds up to exponentially small errors},
  author={Jones, Christopher KRT and Kaper, Tasso J and Kopell, Nancy},
  journal={SIAM J. Appl. Math.},
  volume={27},
  number={2},
  pages={558--577},
  year={1996},
  publisher={SIAM}
}

@incollection{jones37construction,
  title={Construction of the FitzHugh-Nagumo pulse using differential forms},
  author={Jones, CKRT and Kopell, N and Langer, R},
  booktitle={Patterns and dynamics in reactive media},
  pages={101--115},
  year={1991},
  publisher={Springer}
}

@Article{JonesKopell,
  author    = {Jones, CKRT and Kopell, N},
  title     = {Tracking invariant manifolds with differential forms in singularly perturbed systems},
  journal   = {J. Differ. Equations},
  year      = {1994},
  volume    = {108},
  number    = {1},
  pages     = {64--88},
  publisher = {Elsevier},
}

@article{kilpatrick2008traveling,
  title={Traveling pulses and wave propagation failure in inhomogeneous neural media},
  author={Kilpatrick, Zachary P and Folias, Stefanos E and Bressloff, Paul C},
  journal={SIAM J. Appl. Dyn. Syst.},
  volume={7},
  number={1},
  pages={161--185},
  year={2008},
  publisher={SIAM}
}

@Article{kofinas2023latent,
  author  = {Kofinas, Miltiadis Miltos and Bekkers, Erik and Nagaraja, Naveen and Gavves, Efstratios},
  journal = {Adv. Neural. Inf. Process. Syst.},
  title   = {Latent field discovery in interacting dynamical systems with neural fields},
  year    = {2023},
  pages   = {31780--31810},
  volume  = {36},
}

@Article{lang2016multiscale,
  author    = {Lang, Eva},
  journal   = {SIAM J. Appl. Dyn. Syst.},
  title     = {A multiscale analysis of traveling waves in stochastic neural fields},
  year      = {2016},
  number    = {3},
  pages     = {1581--1614},
  volume    = {15},
  publisher = {SIAM},
}

@Article{LijZhangSolo,
  author  = {Zhang, Lijun},
  title   = {Existence and uniqueness of wave fronts in neuronal network with nonlocal post-synaptic axonal and delayed nonlocal feedback connections},
  journal = {Adv. Difference Equ.},
  year    = {2013},
  volume  = {2013},
  number  = {1},
  pages   = {1--15},
}

@Article{luo2025traveling,
  author    = {Luo, Canhuang and Ester, Edward F},
  journal   = {Proc. Natl. Acad. Sci. USA},
  title     = {Traveling waves link human visual and frontal cortex during working memory--guided behavior},
  year      = {2025},
  number    = {30},
  pages     = {e2415573122},
  volume    = {122},
  publisher = {National Academy of Sciences},
}

@Article{Lvwang,
  author  = {Lv, Guangying and Wang, Mingxin},
  title   = {Traveling waves of some integral-differential equations arising from neuronal networks with oscillatory kernels},
  journal = {J. Math. Anal. Appl.},
  year    = {2010},
  volume  = {370},
  number  = {1},
  pages   = {82--100},
}

@article{magpantay2010wave,
  title={Wave fronts in neuronal fields with nonlocal post-synaptic axonal connections and delayed nonlocal feedback connections},
  author={Magpantay, Felicia Maria G and Zou, Xingfu},
  journal={Math. Biosci. Eng.},
  volume={7},
  pages={421--442},
  year={2010}
}

@Article{martinet2017human,
  author    = {Martinet, Louis-Emmanuel and Fiddyment, Grant and Madsen, JR and Eskandar, EN and Truccolo, Wilson and Eden, Uri T and Cash, SS and Kramer, Mark A},
  title     = {Human seizures couple across spatial scales through travelling wave dynamics},
  journal   = {Nat. Commun.},
  year      = {2017},
  volume    = {8},
  pages     = {14896},
  publisher = {Nature Publishing Group},
}

@article{o2021migraine,
  title={Migraine visual aura and cortical spreading depression—Linking mathematical models to empirical evidence},
  author={O’Hare, Louise and Asher, Jordi M and Hibbard, Paul B},
  journal={Vision},
  volume={5},
  number={2},
  pages={30},
  year={2021},
  publisher={MDPI}
}

@Article{oleynik2015iterative,
  author    = {Oleynik, Anna and Ponosov, Arcady and Wyller, John},
  journal   = {Differ. Equat. Dyn. Sys.},
  title     = {Iterative schemes for bump solutions in a neural field model},
  year      = {2015},
  number    = {1},
  pages     = {79--98},
  volume    = {23},
  publisher = {Springer},
}

@article{pavan2022vitro,
  title={In vitro cell models merging circadian rhythms and brain waves for personalized neuromedicine},
  author={Pavan, Barbara and Bianchi, Anna and Botti, Giada},
  journal={Iscience},
  volume={25},
  number={12},
  year={2022},
  publisher={Elsevier}
}

@article{pinto2001spatially,
  title={Spatially structured activity in synaptically coupled neuronal networks: II. Lateral inhibition and standing pulses},
  author={Pinto, David J and Ermentrout, G Bard},
  journal={SIAM J. Appl. Math.},
  volume={62},
  number={1},
  pages={226--243},
  year={2001},
  publisher={SIAM}
}

@article{Pinto2005,
  title={Existence and stability of traveling pulses in a continuous neuronal network},
  author={Pinto, David J and Jackson, Russell K and Wayne, C Eugene},
  journal={SIAM J. Appl. Dyn. Syst.},
  volume={4},
  number={4},
  pages={954--984},
  year={2005},
  publisher={SIAM}
}

@article{PintoandErmentrout-SpatiallyStructuredActivityinSynapticallyCoupledI.TravelingWaves,
  title={Spatially structured activity in synaptically coupled neuronal networks: I. Traveling fronts and pulses},
  author={Pinto, David J and Ermentrout, G Bard},
  journal={SIAM J. Appl. Math.},
  volume={62},
  number={1},
  pages={206--225},
  year={2001},
  publisher={SIAM}
}

@Article{rule2019neural,
  author    = {Rule, Michael E and Schnoerr, David and Hennig, Matthias H and Sanguinetti, Guido},
  journal   = {PLoS Comput Biol.},
  title     = {Neural field models for latent state inference: Application to large-scale neuronal recordings},
  year      = {2019},
  number    = {11},
  pages     = {e1007442},
  volume    = {15},
  publisher = {Public Library of Science San Francisco, CA USA},
}

@Article{Sandstede-EvansFunctions,
  author  = {Sandstede, Bj{\"o}rn},
  title   = {Evans functions and nonlinear stability of traveling waves in neuronal network models},
  journal = {Int. J. Bifurc. Chaos},
  year    = {2007},
  volume  = {17},
  number  = {08},
  pages   = {2693--2704},
}

@Article{sato2022cortical,
  author    = {Sato, Naoyuki},
  journal   = {Sci. Rep.},
  title     = {Cortical traveling waves reflect state-dependent hierarchical sequencing of local regions in the human connectome network},
  year      = {2022},
  number    = {1},
  pages     = {334},
  volume    = {12},
  publisher = {Nature Publishing Group UK London},
}

@Article{shaw2024representing,
  author    = {Shaw, Sage and Kilpatrick, Zachary P},
  journal   = {J. Comput. Neurosci.},
  title     = {Representing stimulus motion with waves in adaptive neural fields},
  year      = {2024},
  number    = {2},
  pages     = {145--164},
  volume    = {52},
  publisher = {Springer},
}

@Article{WilsonCowan1972,
  author   = {Wilson, H R and Cowan, J D},
  title    = {{Excitatory and inhibitory interactions in localized populations of model neurons.}},
  journal  = {Biophys. J.},
  year     = {1972},
  volume   = {12},
  number   = {1},
  pages    = {1--24},
  issn     = {0006-3495},
  doi      = {10.1016/S0006-3495(72)86068-5},
  isbn     = {0006-3495},
  url      = {http://www.sciencedirect.com/science/article/pii/S0006349572860685},
}

@article{zanos2015sensorimotor,
  title={A sensorimotor role for traveling waves in primate visual cortex},
  author={Zanos, Theodoros P and Mineault, Patrick J and Nasiotis, Konstantinos T and Guitton, Daniel and Pack, Christopher C},
  journal={Neuron},
  volume={85},
  number={3},
  pages={615--627},
  year={2015},
  publisher={Elsevier}
}

@article{Zhang-HowDo,
  title={How do synaptic coupling and spatial temporal delay influence traveling waves in nonlinear nonlocal neuronal networks?},
  author={Zhang, Linghai},
  journal={SIAM J. Appl. Dyn. Syst.},
  volume={6},
  number={3},
  pages={597--644},
  year={2007},
  publisher={SIAM}
}

@Article{Zhang-OnStability,
  author  = {Zhang, Linghai},
  journal = {Differ. Integral Equ.},
  title   = {On stability of traveling wave solutions in synaptically coupled neuronal networks},
  year    = {2003},
  pages   = {513--536},
  volume  = {16},
}
\endgroup
\end{document}